\documentclass[12pt]{article}

\usepackage{lipsum}
\usepackage[english]{babel}

\usepackage{pdfpages} 

\usepackage[letterpaper,top=3cm,bottom=3cm,left=3cm,right=3cm,marginparwidth=1.75cm]{geometry}

\usepackage{tikz-cd}
\usepackage{adjustbox}

\usepackage{comment}
\usepackage{amsmath, amssymb, amsthm}
\usepackage{mathtools}
\usepackage{graphicx}
\usepackage{bbm}
\usepackage[colorlinks=true, allcolors=blue]{hyperref}
\usepackage{mathrsfs}
\usepackage{mathabx}
\usepackage{enumerate}
\usepackage{varioref}
\usepackage{placeins}
\RequirePackage{float}

\theoremstyle{definition}
\newtheorem{theorem}{Theorem}[section]
\newtheorem{proposition}[theorem]{Proposition}
\newtheorem{corollary}[theorem]{Corollary}

\newtheorem{definition}[theorem]{Definition}
\newtheorem{lemma}[theorem]{Lemma}
\newtheorem*{remark}{Remark}

\numberwithin{equation}{section}

\counterwithout{equation}{section}

\tikzset{
  symbol/.style={
    draw=none,
    every to/.append style={
      edge node={node [sloped, allow upside down, auto=false]{$#1$}}}
  }
}

\newcommand{\function}[5]{%
  \begin{tikzcd}[
    column sep=2em,
    row sep=1ex,
    ampersand replacement=\&
  ]
  #1\colon \&[-3em]
  #2 \arrow[r] \&
  #3 \\
  \&
  #4 \arrow[u,symbol=\in] \arrow[r,mapsto] \&
  #5 \arrow[u,symbol=\in]
  \end{tikzcd}%
}

\newcommand*\diff{\mathop{}\!\mathrm{d}}

\DeclareMathOperator{\supp}{supp}

\title{\LARGE \textsc{Falconer's problem for dot \\ product on paraboloids}}

\author{{Chun-Kai Tseng}\thanks{This is my Bachelor's thesis at National Taiwan University under the supervision of Professor Chun-Yen Shen.}}
\date{}

\begin{document}

\maketitle

\begin{abstract}
We establish dimensional thresholds for dot product sets associated with compact subsets of translated paraboloids. Specifically, we prove that when the dimension of such a subset exceeds 
\( \frac{5}{4} = \frac{3}{2} - \frac{1}{4} \) in $\mathbb{R}^3$, and \( \frac{d}{2} - \frac{1}{4} - \frac{1}{8d - 4} \) in $\mathbb{R}^d$ for $d\geq 4$, its dot product set has positive Lebesgue measure.

This result demonstrates that if a compact set in \(\mathbb{R}^d\) exhibits a paraboloidal structure, then the usual dimensional barrier of \( \frac{d}{2} \) for dot product sets can be lowered for \( d \geq 3 \). Our work serves as the continuous counterpart of \cite{Discrete_product_sets}, which examines the finite field setting with partial reliance on the extension conjecture. 

The key idea, closely following \cite{Discrete_product_sets}, is to reformulate the dot product set on the paraboloid as a variant of a distance set. This reformulation allows us to leverage state-of-the-art results from the pinned distance problem, as established in \cite{Greatest_pinned_distance_for_d=2} for \( d = 2 \) and \cite{Greatest_pinned_distance_for_d>=3} for higher dimensions. Finally, we present explicit constructions and existence proofs that highlight the sharpness of our results.
\end{abstract}

\tableofcontents

\section{Introduction}

Throughout this paper, we use the notation $\mathcal{P}_d$ to denote the standard paraboloid in $\mathbb{R}^d$, that is
$$
\mathcal{P}_d \coloneqq \{ (\bar{x},|\bar{x}|^2) \colon \bar{x} \in \mathbb{R}^{d-1} \}.
$$
In contrast, when referring to the unit sphere, we write $\mathbb{S}^{d-1} \subset \mathbb{R}^d$. One should be careful with the exponents: in $\mathcal{P}_d$, the exponent $d$ refers to the dimension of the ambient space, while in $\mathbb{S}^{d-1}$, the exponent $(d-1)$ refers to the dimension of the submanifold.

\subsection{Dot product on the paraboloid in the finite field setting}

Motivated by the distance problem, many mathematicians have studied related variants such as simplex configurations, dot products, and other geometric structures. Furthermore, imposing structural constraints on compact sets has led to improved exponent thresholds that guarantee the existence of such configurations. These configuration problems have been explored not only in the Euclidean setting but also in the finite field setting. For instance, Che-Jui Chang, Ali Mohammadi, Thang Pham, and Chun-Yen Shen proved the following result in \cite{Discrete_product_sets}:

\begin{theorem}
Let $F$ be a field, and define the standard paraboloid in $\mathbb{F}^d$ by
$$
\mathcal{P}_d \coloneqq \{ (x_1,\dots,x_{d-1},x_1^2+\cdots+x_{d-1}^2) \colon x_1,\dots, x_{d-1} \in \mathbb{F} \}.
$$
\begin{enumerate}
    \item Let $E \subset \mathcal{P}_d \subset \mathbb{F}_q^{d}$, where $d,q \equiv 3 \pmod{4}$ and $q$ is a prime power. Assume the extension conjecture 
    $$
    \| (f d\sigma) \ \widecheck{} \ \|_{L^{\frac{2d+2}{d-1}}(\mathbb{F}_q^{d-1},dc)} \leq C \|f\|_{L^2(S_r,d\sigma)},
    $$
    holds for every
    $$
    S_r \coloneqq \{x=(x_1,\dots,x_{d-1}) \colon x_1^2+\cdots+x_{d-1}^2=r \} \subset \mathbb{F}_q^{d-1}
    $$
    where $r \neq 0$. Then $|\Pi(E)| \gg q$ whenever $|E| \gg q^{\frac{d}{2}-\frac{1}{2d}}$.
        
    \item Let $E \subset \mathcal{P}_3 \subset \mathbb{F}_p^{3}$, where $p \equiv 3 \pmod{4}$ and $p$ is a prime. Then $|\Pi(E)| \gg p$ whenever $|E| \gg p^{\frac{3}{2}-\frac{1}{4}}$.
    \end{enumerate}
\end{theorem}

Notably, the extension conjecture has been proven in the two-dimensional case, which completes the proof for $d=3$ in the first part and fully establishes the second part.

In \cite{Discrete_product_sets}, the authors leveraged the geometry of the paraboloid to reduce the problem to a distance problem in a lower-dimensional vector space. However, key differences exist between the discrete and continuous settings. For instance, in the discrete case, they applied combinatorial counting arguments and the Cauchy-Schwarz inequality to obtain the estimates:
$$
|\Pi(E)| \geq \frac{|E|^4}{|M(E)|},
$$
and
$$
|\Pi(E)| \geq \frac{|E|^3}{|D(E)|},
$$
where 
$$
M(E) = \{(x,y,z,w) \in E^4 \colon x \cdot y = z \cdot w \}
$$
and
$$
D(E) = \{(x,y,z) \in E^3 \colon x \cdot y = x \cdot z \}.
$$
Although such estimates are crucial in the discrete case, they do not translate directly to the continuous setting.

\subsection{Dot product in Euclidean space}
As in the distance problem, the lowest possible threshold for ensuring that the dot product set has positive Lebesgue measure is $\frac{d}{2}$. The following proposition formalizes this statement by constructing a counterexample.
\begin{proposition} \label{Dot product counterexample}
Let $d\geq 2$. For every $\varepsilon>0$, there exists a compact set $E\subset \mathbb{R}^d$ such that $\dim_H E > \frac{d}{2}-\varepsilon$ and that $|\Pi(E)|=0$.
\end{proposition}

\begin{proof}
Fix $s\in (0,1/2)$. Let $\{q_k\}_{k=1}^\infty$ be a sequence that satisfies $q_{k+1}>q_k^k$. For each $k\in \mathbb{N}$, define
$$
E_{s,q_k} \coloneqq 
\left\{ (x_1,\cdots,x_{d}) \in [0,1]^{d} \colon \exists n_i \in \mathbb{Z} \text{ such that } \left| x_i-\frac{n_i}{q_k}\right| \leq q_k^{-1/s},\, \forall 1\leq i \leq d \right\}
$$
and $E_s \coloneqq \bigcap_{k\in \mathbb{N}} E_{s,q_k}$. One can verify that $\dim_H E_s= sd$; see Theorem 8.15 in Falconer's book \cite{Falconer's_book} for more details. To prove the proposition, we show that for each $s\in (0,1/2)$, $|\Pi(E_s)|=0$.

Note that $\Pi(E_s) \subset \bigcap_{k\in \mathbb{N}} \Pi(E_{s,q_k})$. By direct calculation, for every two points $x=(x_1,\cdots,x_d), y=(y_1,\cdots,y_d) \in E_s$, their dot product is in a  $C_1 \, q_k^{-1/s}$-neighborhood of a lattice point of the form $\frac{n}{q_k^2}$ for all $k \in \mathbb{N}$ where $C_1$ is independent of $k$. Here, $n \in \mathbb{Z}$ lies in an admissible range since $E_s$ is bounded. This implies that $\Pi(E_s)$ is contained in a union of $\leq C_2\, q_k^2$ intervals of length $\leq C_1 \, q_k^{-1/s}$. Here, $C_2$ is also independent of $k$. Therefore, we can conclude that when $s<1/2$, $\Pi(E_s)$ has measure zero. Finally, for every $\varepsilon>0$, we can take $\frac{1}{2}-\frac{\varepsilon}{d}< s < \frac{1}{2}$ and construct the compact subset $E_s$ of $\mathbb{R}^d$ such that
$$
\dim_H E_s =sd > \frac{d}{2} - \varepsilon , \text{ but } |\Pi(E_s)|=0.
$$

\end{proof}

In \cite{Fourier_integral_operators}, Suresh Eswarathasan, Alex Iosevich, and Krystal Taylor proved that a Hausdorff dimension of $\frac{d+1}{2}$ is a sufficient threshold to ensure that the dot product has positive Lebesgue measure. More precisely, Corollary 1.5 in \cite{Fourier_integral_operators} implies the following:
\begin{theorem}\label{Dot product in general}
If $E$ is a compact subset of $\mathbb{R}^d$ with Hausdorff greater than $\frac{d+1}{2}$, then the dot product set $\Pi(E) \coloneqq \{x\cdot y \colon x,y\in E\}$ has positive Lebesgue measure.
\end{theorem}
\begin{remark}
In fact, the authors proved the result for any function $\phi \colon \mathbb{R}^d \times \mathbb{R}^d \to \mathbb{R}$ satisfying the Phong-Stein rotational curvature condition via generalized Radon transform. In particular, $\phi(x,y)=x\cdot y$ satisfies this condition except for the origin, which can be addressed using the pigeonhole principle.
\end{remark}

In \cite{Projection_theorems}, Burak Erdoğan, Derrick Hart, and Alex Iosevich proved the following theorem. Since this theorem plays a crucial role in this research,we will present a detailed proof based on Fourier analysis.
\begin{theorem}\label{projection theorem}
Let $E,F$ be compact subsets in $\mathbb{R}^d$ with Hausdorff dimension $s_E,s_F$. Suppose there exist Frostman measures $\mu_E ,\mu_F$ supported on $E,F$, and a nonnegative number $l_F$  such that for all sufficiently small $\delta>0$,
$$
\mu_F(T_\delta) \lesssim \delta^{s_F-l_F}
$$
for every tube $T_\delta$ of length $\approx 1$, radius $\approx \delta$ that emanates from the origin, and that $s_E+s_F-l_F>d$. Then for $\mu_F$-a.e. $y\in F$, 
$$
\mathcal{L}^1(\Pi^y(E)) = \mathcal{L}^1(\{x\cdot y \colon x\in E\})>0.
$$
\end{theorem}

\begin{remark}
In general, the condition $\dim_H E = s_E$ does not necessarily guarantee the existence of an $s_E$-dimensional Frostman measure supported on $E$. An analogous issue arises for $F$ and $s_F$. More precisely, one requires that $\mathcal{H}^{s_E}(E) > 0$ and $\mathcal{H}^{s_F}(F) > 0$. Nevertheless, this distinction does not affect the dimensional threshold, since in such arguments we typically assume that the compact set under consideration has dimension \emph{strictly} greater than the threshold. This strict inequality provides the necessary flexibility for the analysis.
\end{remark}

\begin{proof}
First of all, by dyadic pigeonholing and rescaling, we may assume $F\subset \{ 
x\in \mathbb{R}^d \colon 1\leq |x| \leq 2 \} $. Let $\phi$ be a smooth cutoff function that is identically $1$ on $E$ and has compact support.

For each $y\in F$, define a measure $\nu_y$ on $\Pi^y(E)$ by the following:
$$
\int_{\mathbb{R}} g(s) \diff \nu_y(s) = \int_{\mathbb{R}^d} g(x\cdot y) \diff \mu_E(x).
$$
In other words, $\nu_y$ is the push-forward measure of $\mu_E$ under the map 
$$
\function{\Pi^y}{E}{\Pi^y(E)}{x}{x\cdot y}
$$
Observe that
$$
\begin{aligned}
\widehat{\nu_y}(t)
&= \int_{\mathbb{R}} e^{-2\pi i t s} \diff \nu_y(s) 
=\int_{\mathbb{R}^d} e^{-2\pi i t(x\cdot y)} \diff \mu_E(x)
=\int_{\mathbb{R}^d} e^{-2\pi i x\cdot (ty)} \diff \mu_E(x)
=\widehat{\mu_E}(ty)
\end{aligned}
$$

Below, we will show that
$$
\int_{\mathbb{R}^d} \int_{\mathbb{R}} |\widehat{\nu_y}(t)|^2 \diff t \diff \mu_F(y) < \infty.
$$
If this is true, then for $\mu_F$-a.e.$\ y\in F$, $\widehat{\nu_y} \in L^2$, which implies that $\nu_y \in L^2$. In other words, for $\mu_F$-a.e.$\ y\in F$, the  probability measure $\nu_y$ has an $L^2$-density. Remember that $\supp(\nu_y) \subset \Pi^y(E)$. Thus, for $\mu_F$-a.e.$\ y\in F$, $\mathcal{L}^1(\Pi^y(E))>0$, which is the statement of the theorem. Observe that we can assume the range of integration with respect to $t$ is $\{ t\colon |t|\geq 4 \}$. This is because $|\widehat{\nu_y}(t)|^2 \leq 1$, $\{ t\colon |t|\leq 4 \}$ has finite measure, and $\mu_F$ is a probability measure.

We now estimate the integral:
$$
\begin{aligned}
\int_{\mathbb{R}^d} \int_{|t|\geq 4} |\widehat{\nu_y}(t)|^2 \diff t \diff \mu_F(y)
&= \int_{\mathbb{R}^d} \int_{|t|\geq 4} |\widehat{\mu_E}(ty)|^2 \diff t \diff \mu_F(y)\\
&= \int_{\mathbb{R}^d} \int_{|t|\geq 4} |\widehat{\mu_E} * \widehat{\phi} (ty)|^2 \diff t \diff \mu_F(y)\\
&= \int_{\mathbb{R}^d} \int_{|t|\geq 4} \left| \int_{\mathbb{R}^d} \widehat{\mu_E}(\xi) \widehat{\phi} (ty-\xi) \diff \xi \right|^2 \diff t \diff \mu_F(y)\\
&\lesssim \int_{\mathbb{R}^d} \int_{|t|\geq 4} \int_{\mathbb{R}^d} |\widehat{\mu_E}(\xi)|^2 |\widehat{\phi}(ty-\xi)| \diff \xi \diff t \diff \mu_F(y)
\end{aligned}
$$
where the last inequality follows from the Cauchy-Schwarz inequality and the bound $\int |\widehat{\phi}(ty-\xi)| \diff \xi \leq C$, where $C$ is a constant independent of $t$ and $y$. Given the form of the integrand, it is natural to decompose the integral based on the relative magnitudes of $|\xi|$ and $|ty-\xi|$. Moreover, since $\widehat{\phi}$ is a Schwartz function, it satisfies the decay estimate
$$
|\widehat{\phi}(\eta)| \leq C_N |\eta|^{-N}
$$
for any $N>0$. This decay property is crucial in our analysis, as we shall see shortly.
~\\

\noindent \textbf{\underline{Region 1: $|\xi| \leq 2$}}
~\\\\
In this case, $|ty-\xi| \geq |t|\cdot |y|-|\xi| \gtrsim |t|$. Using the estimates $|\widehat{\mu_E}(\xi)|\leq 1$ and $|\widehat{\phi}(ty-\xi)|\lesssim |ty-\xi|^{-N} \lesssim t^{-N}$, we obtain
$$
\begin{aligned}
\int_{\mathbb{R}^d} \int_{|t|\geq 4} \int_{|\xi|\leq 2} |\widehat{\mu_E}(\xi)|^2 |\widehat{\phi}(ty-\xi)| \diff \xi \diff t \diff \mu_F(y)
\lesssim \int_{\mathbb{R}^d} \int_{|t|\geq 4} \int_{|\xi|\leq 2} t^{-N} \diff \xi \diff t \diff \mu_F(y) <\infty
\end{aligned}
$$

~\\
\noindent \textbf{\underline{Region 2: $|\xi| \geq 2,\, |ty-\xi|\leq 1$}}
~\\\\
In this case, the only available estimate for \( \widehat{\phi}(ty-\xi) \) is \( |\widehat{\phi}(ty-\xi)| \lesssim 1 \). Observe that in this region, we have $|ty| \approx |\xi|$, which implies $|t| \approx |\xi|$ because we assume $F$ is contained in the annulus of outer radius $2$ and inner radius $1$. Hence, we can estimate the integral:
$$
\begin{aligned}
&\int_{|ty-\xi|\leq 1} \int_{|t|\geq 4} \int_{|\xi|\geq 2} |\widehat{\mu_E}(\xi)|^2 |\widehat{\phi}(ty-\xi)| \diff \xi \diff t \diff \mu_F(y)\\
&\lesssim \int_{|\xi|\geq 2} |\widehat{\mu_E}(\xi)|^2 (\mu_F \times \mathcal{L}^1) \{ (y,t)\colon |ty-\xi| \leq 1 \} \diff \xi
\end{aligned}
$$
Note that
$$
|ty-\xi| \leq 1 \; \Longrightarrow \; \left| y-\frac{|\xi|}{t}\frac{\xi}{|\xi|} \right| 
\leq \frac{1}{|t|} \lesssim \frac{1}{|\xi|}.
$$
Fixing $\xi$, we now analyze the possible range of $y$. Note that $\xi/|\xi|$ is the unit vector in the direction of $\xi$ and that $|\xi|/t$ is a scalar, possibly negative, $\approx 1$. Since \( t \) varies, \( y \) lies within a distance \( \lesssim |\xi|^{-1} \) of either the line segment joining \( c_1 \xi/|\xi| \) and \( c_2 \xi/|\xi| \) or the one joining \( -c_1 \xi/|\xi| \) and \( -c_2 \xi/|\xi| \), where \( c_1 < c_2 \) are the lower and upper bounds for the ratio \( |\xi|/|t| \). Therefore, the possible range of $y$ is a union of two tubes of radius $\approx |\xi|^{-1}$ and length $\approx 1$ with directions $\xi/|\xi|$ and $-\xi/|\xi|$. Now, we fix $\xi,y$, by the condition $|ty-\xi|\leq 1$, we see that the range of $t$ is contained in an interval of length $\lesssim 1$. Therefore,
$$
\begin{aligned}
&\int_{|\xi|\geq 2} |\widehat{\mu_E}(\xi)|^2 (\mu_F \times \mathcal{L}^1) \{ (y,t)\colon |ty-\xi| \leq 1 \} \diff \xi\\
&\lesssim \int_{|\xi|\geq 2} |\widehat{\mu_E}(\xi)|^2 
\left[ \mu_F(T_{|\xi|^{-1}}(\xi)) + \mu_F(T_{|\xi|^{-1}}(-\xi)) \right]
\diff \xi\\
&\lesssim \int_{|\xi|\geq 2} |\widehat{\mu_E}(\xi)|^2 |\xi|^{-s_F+l_F} \diff \xi<\infty
\end{aligned}
$$
provided that $s_F-l_F>d-s_E$, which is true by the assumption of the theorem.

~\\
\noindent \textbf{\underline{Region 3: $\forall m \in \mathbb{N}$, $2^m \leq |\xi| \leq 2^{m+1} $, and $\forall \l\in \mathbb{N}\cup \{0\} \colon \l\leq m-3$, $2^{\l} \leq |ty-\xi| \leq 2^{\l+1}$  }}
~\\\\
This region consists of many parts associated with different $m,\l$. We first fix $\l \in \mathbb{N}\cup \{0\}$ and integrate $\xi$ over the region $\{\xi \colon |\xi| \geq 2^{\l+3} \}$. Finally, we sum over all $\l \in \mathbb{N}\cup \{0\}$. 

Under the assumption of $m\geq \l+3$, we see that
$$
|ty-\xi| \leq 2^{\l+1} \leq \frac{1}{4} \cdot 2^m \leq \frac{1}{4} |\xi|.
$$
This implies that $|ty|\approx |\xi|$ and that $|t| \approx |\xi|$. Then we can estimate the integral of this region
$$
\begin{aligned}
\text{integral}
&= \sum_{\l \geq 0} \int_{|t|\geq 4} \int_{|\xi|\geq 2^{\l+3}} \int_{2^{\l} \leq |ty-\xi| \leq 2^{\l+1}} 
|\widehat{\mu_E}(\xi)|^2 |\widehat{\phi}(ty-\xi)| \diff \mu_F(y) \diff \xi \diff t \\
&\lesssim \sum_{\l \geq 0} \int_{|\xi|\geq 2^{\l+3}}  |\widehat{\mu_E}(\xi)|^2 \cdot 2^{-\l N} \cdot
(\mu_F \times \mathcal{L}^1) \{ (y,t)\colon |ty-\xi| \leq 2^{\l+1} \} \diff \xi\\
\end{aligned}
$$
By the same argument as before, we see that
$$
|ty-\xi| \leq 2^{\l+1} \Longrightarrow \left| y-\frac{|\xi|}{t} \frac{\xi}{|\xi|} \right| \lesssim 2^{\l+1} |\xi|^{-1}
$$
and that
$$
\begin{aligned}
(\mu_F \times \mathcal{L}^1) \{ (y,t)\colon |ty-\xi| \leq 2^{\l+1} \}
&\lesssim 2^{\l} \left[ \mu_F(T_{2^{\l+1}|\xi|^{-1}}(\xi)) +\mu_F(T_{2^{\l+1}|\xi|^{-1}}(-\xi)) \right]
\end{aligned}
$$
where the factor $2^{\l}$ is an upper bound of the length of the range of $t$ when both $\xi,y$ are fixed. Thus,
$$
\begin{aligned}
\text{integral}
&\lesssim \sum_{\l\geq 0} 2^{-\l N} \int_{|\xi|\geq 2^{\l+3}} |\widehat{\mu_E}(\xi)|^2 \cdot 2^{\l} \cdot
\left[ \mu_F(T_{2^{\l+1}|\xi|^{-1}}(\xi)) +\mu_F(T_{2^{\l+1}|\xi|^{-1}}(-\xi)) \right] \diff \xi\\
&\lesssim \sum_{\l\geq 0} 2^{-\l (N-1)} \int_{|\xi|\geq 2^{\l+3}} |\widehat{\mu_E}(\xi)|^2 \left( 2^{\l+1} \right)^{s_F-\l_F} |\xi|^{-s_F + \l_F} \diff \xi\\
&\lesssim \int_{|\xi|\geq 8} |\widehat{\mu_E}(\xi)|^2  |\xi|^{-s_F + \l_F} \diff \xi \cdot \left( \sum_{\l\geq 0} 2^{-\l(N-1+\l_F-s_F)} \right) < \infty
\end{aligned}
$$
provided that $N$ is large enough.

~\\
\noindent \textbf{\underline{Region 4: $\forall m \in \mathbb{N}$, $2^m \leq |\xi| \leq 2^{m+1} $, $|ty-\xi| \geq 2^{m-2}$, and $|t| \leq 2^m$}}
~\\\\

Observe that in this region, we have
$$
|\widehat{\phi}(ty-\xi)| \lesssim |ty-\xi|^{-N} \lesssim 2^{-mN}.
$$
Hence, the integral of this region can be estimated as below:
$$
\begin{aligned}
\text{integral}
&= \sum_{m\geq 1} \int_{4\leq |t|\leq 2^m} \int_{2^m \leq |\xi| \leq 2^{m+1}} \int_{|ty-\xi| \geq 2^{m-2}}
|\widehat{\mu_E}(\xi)|^2 |\widehat{\phi}(ty-\xi)|  \diff \mu_F(y) \diff \xi \diff t\\
&\lesssim \sum_{m\geq 1} \int_{4\leq |t|\leq 2^m} \int_{2^m \leq |\xi| \leq 2^{m+1}} \int_{\mathbb{R}^d} 
|\widehat{\mu_E}(\xi)|^2 \cdot 2^{-mN}  \diff \mu_F(y) \diff \xi \diff t\\
&= \sum_{m\geq 1} 2^{-mN} \int_{4\leq |t|\leq 2^m} \int_{2^m \leq |\xi| \leq 2^{m+1}} 
|\widehat{\mu_E}(\xi)|^2 \diff \xi \diff t\\
&\lesssim \sum_{m\geq 1} 2^{-mN} \cdot 2^m 
\int_{2^m \leq |\xi| \leq 2^{m+1}} |\widehat{\mu_E}(\xi)|^2 \diff \xi\\
&\lesssim \sum_{m\geq 1} 2^{-m(N-1)} 2^{md} <\infty
\end{aligned}
$$
provided that $N$ is chosen large enough.

~\\
\noindent \textbf{\underline{Region 5: $\forall m \in \mathbb{N}$, $2^m \leq |\xi| \leq 2^{m+1} $, $|ty-\xi| \geq 2^{m-2}$, and $|t| \geq 2^m$}}
~\\\\

Note that under the assumption, we have
$$
|t| \lesssim |ty-\xi| + |\xi| \lesssim |ty-\xi|
$$
because $|\xi|\leq 2^{m+1} = 8\cdot 2^{m-2} \leq 8 |ty-\xi|$. So,
$$
\begin{aligned}
\text{integral}
&= \sum_{m\geq 1} \int_{|t|\geq 2^m} \int_{2^m \leq |\xi| \leq 2^{m+1}} \int_{|ty-\xi| \geq 2^{m-2}}
|\widehat{\mu_E}(\xi)|^2 |\widehat{\phi}(ty-\xi)|  \diff \mu_F(y) \diff \xi \diff t\\
&\lesssim \sum_{m\geq 1} \int_{|t|\geq 2^m} \int_{2^m \leq |\xi| \leq 2^{m+1}} \int_{|ty-\xi| \geq 2^{m-2}}
|\widehat{\mu_E}(\xi)|^2 \cdot \left[ |\widehat{\phi}(ty-\xi)| \cdot |ty-\xi|^2 \right] \cdot |t|^{-2} \diff \mu_F(y) \diff \xi \diff t\\
&\lesssim \sum_{m\geq 1} \int_{|t|\geq 2^m} \int_{2^m \leq |\xi| \leq 2^{m+1}} \int_{\mathbb{R}^d} |\widehat{\mu_E}(\xi)|^2 \cdot 2^{-mN} \diff \mu_F(y) |t|^{-2} \diff \xi \diff t\\
&= \sum_{m\geq 1} 2^{-mN} \int_{|t|\geq 2^m} \int_{2^m \leq |\xi| \leq 2^{m+1}}  |\widehat{\mu_E}(\xi)|^2 \cdot |t|^{-2} \diff \xi \diff t\\
&\lesssim \sum_{m\geq 1} 2^{-mN} \left[ \int_{|t|\geq 2^m} |t|^{-2} \diff t  \right] \cdot
\left[ \int_{2^m \leq |\xi| \leq 2^{m+1}} |\widehat{\mu_E}(\xi)|^2 \diff \xi \right]\\
&\lesssim \sum_{m\geq 1} 2^{-m(N-d)} <\infty
\end{aligned}
$$
if $N$ is chosen properly.

\end{proof}

\begin{remark}
Note that we can always take $l_F=1$, which gives an alternative proof of Theorem \ref{Dot product in general}.
\end{remark}

Burak Erdoğan, Derrick Hart, and Alex Iosevich also established the following corollary in the same paper \cite{Projection_theorems}:
\begin{corollary}\label{Sphere dot product}
Let a compact set
$$
E\subset \mathbb{S}^{d-1} = 
\left\{ x = (x_1,\cdots,x_d) \in \mathbb{R}^d \colon \sqrt{x_1^2+x_2^2+\cdots+x_d^2}=1 \right\}
$$
be of Hausdorff dimension greater than $\frac{d}{2}$. Let $\mu_E$ be a Frostman measure on $E$. Then
$$
\mathcal{L}^1( \{ x\cdot y \colon x\in E \} )>0
$$
for $\mu_E-$a.e. $y\in E$. In particular, we have $|\Pi(E)|>0$.
\end{corollary}

\begin{remark}
Corollary 1.1.1 in \cite{Projection_theorems} originally concerned distances rather than dot products. However, the authors established the distance result through an analysis of the dot product. This follows from the observation that for two points $x, y$ on the sphere of radius $r$ centered at the origin,
$$
|x-y|^2 = |x|^2 + |y|^2 -2x\cdot y= 2r^2 -2x\cdot y.
$$
\end{remark}

\begin{remark}
The results in \cite{Projection_theorems} provide a stronger, pinned version: for admissible sets $E, F$ and measures $\mu_E, \mu_F$, it holds that for $\mu_F$-almost every $y\in F$, $|\Pi^y(E)|>0$. In other words, there exist many points $x \in E$ such that $|\Pi^x(E)|>0$, which is clearly a stronger result than merely proving $|\Pi(E)|>0$.
\end{remark}

\begin{remark}
Based on the results of \cite{Discrete_product_sets}, one may conjecture whether the $\frac{d}{2}$ exponent for dot products on paraboloids can be improved in the continuous setting. However, if we only apply Theorem \ref{projection theorem}, we can never lower the $\frac{d}{2}$ threshold, as $l$ is always nonnegative, which forces $2s>d$. Moreover, if the $\frac{d}{2}$ threshold could be surpassed, it would contrast with the case of dot product on the sphere $\mathbb{S}^{d-1}$, as shown in Corollary \ref{Sphere dot product}.
\end{remark}

\subsection{State-of-the-art pinned distance results}
In this subsection, we present the latest results on the pinned distance problem, which serve as essential components for our main result.

For $d=2$, the best known result is due to Larry Guth, Alex Iosevich, Yumeng Ou, and Hong Wang \cite{Greatest_pinned_distance_for_d=2}:

\begin{theorem} \label{pinned distance for d=2}
If $E\subset \mathbb{R}^2$ is a compact subset of dimension larger than $\frac{5}{4}$, then there exists a point $x\in E$ such that $\Delta^x(E) \coloneqq \{|x-y| \colon y\in E\}$ has positive Lebesgue measure.
\end{theorem}

As for $d\geq 3$, the best known result is given by Xiumin Du, Yumeng Ou, Kevin Ren, and Ruixiang Zhang \cite{Greatest_pinned_distance_for_d>=3}:
\begin{theorem} \label{pinned distance for d>=3}
If $d\geq 3$ and $E\subset \mathbb{R}^d$ is a compact subset of dimension larger than $\frac{d}{2}+\frac{1}{4}-\frac{1}{8d+4}$, then there exists a point $x\in E$ such that $\Delta^x(E) $ has positive Lebesgue measure.
\end{theorem}

The above two results were derived using the decoupling method. Moreover, one can notice that at the beginning of the two papers, the authors considered two subsets $E_1,E_2$ of $E$ that remain large in dimension and are separated by a distance $\gtrsim 1$. Therefore, we obtain a stronger version of these pinned distance theorems:

\begin{theorem} \label{state-of-the-art pinned distance we use}
\begin{enumerate}
\item If $E_1, E_2$ are two compact subsets of $\mathbb{R}^2$ whose dimensions are both greater than $\frac{5}{4}$, then there exists $x\in E_2$ such that $\Delta^{x}(E_1)$ has positive Lebesgue measure.
\item If \( d\geq 3 \), and \( E_1, E_2 \) are two compact subsets of $\mathbb{R}^d$ whose dimensions are both greater than $\frac{d}{2}+\frac{1}{4}-\frac{1}{8d+4}$, then there exists $x\in E_2$ such that $\Delta^{x}(E_1)$ has positive Lebesgue measure.
\end{enumerate}
\end{theorem}

\begin{proof}
If \( E_1\cap E_2 \) has dimension greater than the threshold in Theorem \ref{pinned distance for d=2} or Theorem \ref{pinned distance for d>=3}, then we can apply the corresponding theorem directly. Otherwise, both \( E_1 \backslash E_2 \) and \( E_2 \backslash E_1 \) have dimension larger than the threshold. In this case, applying the pigeonhole principle, we obtain two compact subsets \( E_1', E_2' \) of \( E_1,E_2 \), respectively, that are separated and have sufficiently large dimensions. By applying the arguments in Theorem \ref{pinned distance for d=2} or Theorem \ref{pinned distance for d>=3} to \( E_1' \) and \( E_2' \), we conclude that there exists a point \( x\in E_2' \) such that \( \Delta^x(E_1') \) has positive Lebesgue measure, which establishes the desired result.
\end{proof}

\subsection{The main result}

The two theorems below are our main results:
\begin{theorem}[Case $d=3$]\label{main d=3}
For every $a \in \mathbb{R}^3$, let
$$
\mathcal{P}_3 \coloneqq \{ (\bar{x},|\bar{x}|^2) \colon \bar{x} \in \mathbb{R}^{2} \} \subset \mathbb{R}^3
$$
be the standard paraboloid, and consider its translation $a+\mathcal{P}_3$.  
If $E \subset a+\mathcal{P}_3$ is compact and
$$
\dim_H E > \tfrac{5}{4} = \tfrac{3}{2}-\tfrac{1}{4},
$$
then the dot product set $\Pi(E)$ has positive Lebesgue measure.
\end{theorem}

\begin{theorem}[Case $d\geq 4$]\label{main d>=4}
For every $a \in \mathbb{R}^d$ with $d \geq 4$, let
$$
\mathcal{P}_d \coloneqq \{ (\bar{x},|\bar{x}|^2) \colon \bar{x} \in \mathbb{R}^{d-1} \} \subset \mathbb{R}^d
$$
be the standard paraboloid, and consider its translation $a+\mathcal{P}_d$.  
If $E \subset a+\mathcal{P}_d$ is compact and
$$
\dim_H E > \tfrac{d}{2}-\tfrac{1}{4}-\tfrac{1}{8d-4},
$$
then the dot product set $\Pi(E)$ has positive Lebesgue measure.
\end{theorem}

\begin{remark}
The results for \( d=3 \) and \( d \geq 4 \) differ due to a distinction in the best-known pinned distance bounds: for \( d=2 \), the bound is \( \frac{5}{4} \), whereas for \( d \geq 3 \), it is given by \( \frac{d}{2}+\frac{1}{4}-\frac{1}{8d+4} \).
\end{remark}

\begin{remark}
In fact, the above results can be strengthened: for every compact subset $E$ of the translated paraboloid $a+\mathcal{P}_d$ whose dimension exceeds the thresholds given in the previous theorems, there exists a point $x\in E$ such that the pinned dot product set $\Pi^x(E)$ has positive Lebesgue measure.
\end{remark}

\section{The key observation and a few lemmas}
\subsection{The key observation}\label{key observation}

Let $a=(\bar{a},a_d) \in \mathbb{R}^d$ represent a translation vector where $\bar{a} \in \mathbb{R}^{d-1}, a_d \in \mathbb{R}$. Recall the definition of the standard paraboloid:
$$
\mathcal{P}_d =\{ (\bar{x},|\bar{x}|^2) \colon \bar{x}\in \mathbb{R}^{d-1} \}.
$$ 
We consider the dot product on the translated paraboloid $a+ \mathcal{P}_d$. By the definition of translation, we have
$$
a+\mathcal{P}_d =\{ (\bar{x}+\bar{a},|\bar{x}|^2+a_d) \colon \bar{x}\in \mathbb{R}^{d-1} \}.
$$

\begin{figure}[!ht]
    \centering
    \includegraphics[scale=0.5]{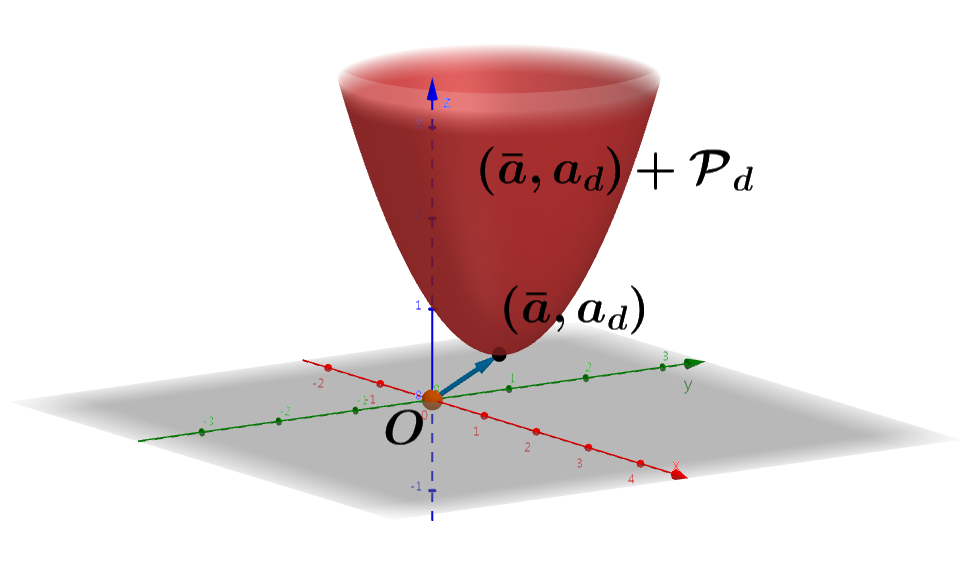}
    \caption{Translated paraboloid $a+\mathcal{P}_d$}
    \label{translated paraboloid}
\end{figure}
\FloatBarrier

For any two points $(\bar{x}+\bar{a},|\bar{x}|^2+a_d) , (\bar{y}+\bar{a},|\bar{y}|^2+a_d)$ on $ a+\mathcal{P}_d$, suppose their dot product is $t$. Then, we compute:
$$
\begin{aligned}
    t&= (\bar{x}+\bar{a},|\bar{x}|^2+a_d) \cdot (\bar{y}+\bar{a},|\bar{y}|^2+a_d)\\
    &= \bar{x}\cdot \bar{y} + \bar{a}\cdot \bar{x}
    +\bar{a}\cdot \bar{y}+ |\bar{a}|^2 + |\bar{x}|^2 |\bar{y}|^2 + a_d |\bar{x}|^2 +a_d |\bar{y}|^2 + a_d^2\\
    &= (|\bar{x}|^2 + a_d)|\bar{y}|^2 + \bar{y} \cdot (\bar{x}+\bar{a}) + (\bar{a}\cdot \bar{x}+|\bar{a}|^2 + a_d |\bar{x}|^2 + a_d^2) 
\end{aligned}
$$
Completing the square gives
\begin{equation}\label{dot product and distance}
\begin{aligned}
    t&= (|\bar{x}|^2 + a_d)|\bar{y}|^2 +2 \bar{y} \cdot \frac{\bar{x}+\bar{a}}{2} + (\bar{a}\cdot \bar{x}+|\bar{a}|^2 + a_d |\bar{x}|^2 + a_d^2)\\
    &= (|\bar{x}|^2+a_d) \left| \bar{y} + \frac{\bar{x}+\bar{a}}{2(|\bar{x}|^2+a_d)}  \right|^2 + h(a,\bar{x})\\
    &= (|\bar{x}|^2+a_d) \left| -\frac{\bar{x}+\bar{a}}{2(|\bar{x}|^2+a_d)} -\bar{y} \right|^2 + h(a,\bar{x})\\
\end{aligned}
\end{equation}
where
$$
h(a,\bar{x}) = (\bar{a}\cdot \bar{x}+|\bar{a}|^2 + a_d |\bar{x}|^2 + a_d^2) - \frac{|\bar{x}+\bar{a}|^2}{4(|\bar{x}|^2+a_d)}.
$$
Note that the above computation requires $|\bar{x}|^2+a_d \neq 0$.

Now, fix the translation vector $a=(\bar{a},a_d)$ and the translated paraboloid $a+\mathcal{P}_d$. For any $\bar{x} \in \mathbb{R}^{d-1}$ such that $|\bar{x}|^2+a_d \neq 0$, it determines both $-\frac{\bar{x}+\bar{a}}{2(|\bar{x}|^2+a_d)}$ and $h(a,\bar{x})$. If $E$ is a compact subset of $a+\mathcal{P}_d$, we obtain
\begin{equation}\label{positive measure iff}
\Pi^{(\bar{x},|\bar{x}|^2)+a} (E)
= (|\bar{x}|^2 + a_d ) (\Delta^2)^{-\frac{\bar{x}+\bar{a}}{2(|\bar{x}|^2+a_d)}} (\bar{E}) + h(a,\bar{x})
\end{equation}
where $(\bar{x},|\bar{x}|^2)+a$ is a point on $a+\mathcal{P}_d$, $\Delta^2$ refers to the configuration of distance square, and $\bar{E}=\{\bar{x} \colon (\bar{x}+\bar{a},|\bar{x}|^2+a_d) \in E\}$.

Remember that we have already fixed $a,\bar{x}$, so
$$
|\Pi^{(\bar{x},|\bar{x}|^2)+a} (E)|>0 \;\;\;\; \Longleftrightarrow \;\;\;\;  |\Delta^{-\frac{\bar{x}+\bar{a}}{2(|\bar{x}|^2+a_d)}} (\bar{E})|>0
$$
because the distance square set has positive Lebesgue measure if and only if the distance set does. To conclude, we have ``generically" transformed the dot product problem on a translated paraboloid into a pinned distance problem in one lower dimension.

\subsection{Transformation maps}
From \eqref{positive measure iff}, we observe the significance of the mapping $\bar{x} \longmapsto -\frac{\bar{x}+\bar{a}}{2(|\bar{x}|^2+a_d)}$. For convenience, we introduce the following definition.

\begin{definition}
For $a=(\bar{a},a_d) \in \mathbb{R}^d$, the transformation map associated with $a$, denoted as $T_a$, is defined by
$$
\function{T_a}{\mathbb{R}^{d-1}\backslash \{\bar{x} \colon |\bar{x}|^2+a_d =0 \}}{\mathbb{R}^{d-1}}{\bar{x}}{-\frac{\bar{x}+\bar{a}}{2(|\bar{x}|^2+a_d)}}
$$
\end{definition}

In order to apply the pinned distance results, we need to verify that transformation maps preserve the Hausdorff dimension of $\bar{E}$. We begin with the following lemma.

\begin{lemma}\label{not lose dimension}
$U$ and $V$ are open sets in $\mathbb{R}^d$. Suppose $f\colon U\subset \mathbb{R}^d \to V \subset \mathbb{R}^d$ is $C^1$ and its Jacobian does not vanish on its domain $U$. Then $\dim_H f(K)= \dim_H K$ for all compact sets $K$ in $U$.
\end{lemma}

\begin{proof}
Since $f$ is locally Lipschitz, we immediately obtain $\dim_H f(K)\leq \dim_H K$. The majority of our proof focuses on establishing the converse inequality. Since the Jacobian is nonzero, the Inverse Function Theorem ensures that for each $x\in K$, there exists an open ball $B_x$ centered at $x$ such that $B_x \subset 2B_x \subset U,\, f(2B_x) \subset V$ and that
$$
f|_{2B_x} \colon 2B_x \longrightarrow f(2B_x)
$$
is a $C^1$-diffeomorphism. Note that $\{ B_x \}_{x\in K}$ is an open cover of $K$ and that by compactness, there exists a finite cover $\{ B_{x_i} \}_{i=1}^N$ of $K$. Thus, there exists at least one chart whose closure intersects $K$ in a way that preserves the dimension of $K$. Without loss of generality, we assume $\dim_H (K\cap \overline{B_{x_1}})=\dim_H K$.

Observe that $K\cap \overline{B_{x_1}}$ is now in the domain of the $C^1$-diffeomorphism $f|_{2B_{x_1}}$. Since a $C^1$-diffeomorphism preserves dimension, we have
$$
\dim_H f(K) \geq \dim_H f(K\cap \overline{B_{x_1}}) = \dim_H (K\cap \overline{B_{x_1}}) = \dim_H K.
$$
\end{proof}

Thus, we proceed to compute the Jacobian determinant of the transformation maps.
\begin{proposition} \label{Jacobian of transformation maps}
For $a=(\bar{a},a_d) \in \mathbb{R}^d$, we have
$$
\operatorname{Jac} T_a = 2 \left( \frac{-1}{2(|\bar{x}|^2+a_d)} \right)^d
\left[ \sum_{i=1}^{d-1}(x_i+a_i)^2 -a_1^2 -\cdots -a_{d-1}^2 -a_d \right].
$$
\end{proposition}

\begin{proof}
Recall the definition 
$$
T_a \colon \bar{x} \longmapsto -\frac{\bar{x}+\bar{a}}{2(|\bar{x}|^2+a_d)} = \bar{y}.
$$
So, for each $1\leq i \leq d-1$,
$$
y_i = \frac{-x_i-a_i}{2(x_1^2+ \cdots +x_{d-1}^2 +a_d)}.
$$
By direct calculation, for $1\leq i\neq j \leq d-1$,
$$
\frac{\partial y_i}{\partial x_i}
= \frac{-1}{2(|\bar{x}|^2+a_d)} + \frac{2x_i(x_i+a_i)}{2(|\bar{x}|^2+a_d)^2}
$$
and
$$
\frac{\partial y_i}{\partial x_j}
= \frac{2x_j(x_i+a_i)}{2(|\bar{x}|^2+a_d)^2}.
$$
If we let
$$
a_{ij} \coloneqq \frac{x_j(x_i+a_i)}{(|\bar{x}|^2+a_d)^2},\;\;\;
z \coloneqq \frac{-1}{2(|\bar{x}|^2+a_d)},
$$
we can write 
$$
\frac{\partial y_i}{\partial x_i} = a_{ii}+z,\quad \frac{\partial y_i}{\partial x_j} = a_{ij}.
$$
Also note that in the matrix $(a_{ij})_{1\leq i,j\leq d-1}$, every two rows are linearly dependent. Therefore, we can apply row operations to calculate the Jacobian determinant:
$$
\det
\begin{pmatrix}
    a_{11}+z &a_{12} &\cdots &a_{1,d-1}\\
    a_{21} &a_{22}+z &\cdots &a_{2,d-1}\\
    \vdots &\vdots &\ddots &\vdots \\
    a_{d-1,1} &a_{d-1,2} &\cdots &a_{d-1,d-1}+z
\end{pmatrix}.
$$
By applying row operations (adding a multiple of the first row to other rows), we can turn the lower right $(d-2) \times (d-2)$ matrix into $z I_{d-2}$, transform the upper right $1\times(d-2)$ matrix into the zero vector, and turn the $(1,1)$-component into
$$
z + a_{11} + \sum_{i=2}^{d-1} a_{1,i} \frac{a_{i,1}}{a_{11}}
$$
under the assumption $a_{11} \neq 0$. Therefore, we obtain
$$
\begin{aligned}
    \operatorname{Jac} T_a &= z^{d-2} \left( z + a_{11} + \sum_{i=2}^{d-1} a_{1,i} \frac{a_{i,1}}{a_{11}} \right)\\
    &= z^{d-2} \left( \frac{-1}{2(|\bar{x}|^2+a_d)} + \frac{(x_1+a_1)2x_1}{2(|\bar{x}|^2+a_d)^2} + \sum_{i=2}^{d-1} \frac{(x_1+a_1)2x_i}{2(|\bar{x}|^2+a_d)^2} \cdot\frac{x_i+a_i}{x_1+a_1} \right)\\
    &= \frac{z^{d-2}}{2(|\bar{x}|^2+a_d)^2} \left[ -|\bar{x}|^2-a_d +\sum_{i=1}^{d-1} 2x_i(x_i+a_i) \right]\\
    &= 2z^d \left[ \sum_{i=1}^{d-1} (x_i+a_i)^2 -a_1^2 -\cdots -a_{d-1}^2 -a_d \right]\\
    &= 2 \left( \frac{-1}{2(|\bar{x}|^2+a_d)} \right)^d
    \left[ \sum_{i=1}^{d-1}(x_i+a_i)^2 -a_1^2 -\cdots -a_{d-1}^2 -a_d \right]
\end{aligned}
$$
provided that $a_{11} \neq 0$. Finally, we verify that the formula also holds for $a_{11}= 0$ by direct computation.

\end{proof}

\subsection{A technical lemma concerning Theorem \ref{projection theorem}}
This subsection presents a lemma that plays a crucial role in Section 3.

\begin{lemma} \label{technical lemma}
Suppose $M\subset \mathbb{R}^d$ is a compact $(d-1)$-dimensional hypersurface that does not contain the origin and there exists a compact subset $K$ of $
\mathbb{S}^{d-1}$ and a $C^1$-function $r \colon K\subset \mathbb{S}^{d-1} \to \mathbb{R}_{>0}$ such that $M$ can be described by the bijective map $\imath$
$$
\function{\imath}{K\subset \mathbb{S}^{d-1}}{M}{e}{r(e) e}
$$
Then, there exists a constant $N>0$ such that for all $e\in \mathbb{S}^{d-1}$ and for all sufficiently small $\delta>0$, the intersection of $M$ with any tube $T_\delta(e)$ of radius $\delta$, emanating from the origin in the direction of $e$, can be covered by at most $N$ balls of radius $2\delta$. 
\end{lemma}

\begin{figure}[!ht]
    \centering
    \includegraphics[scale=0.5]{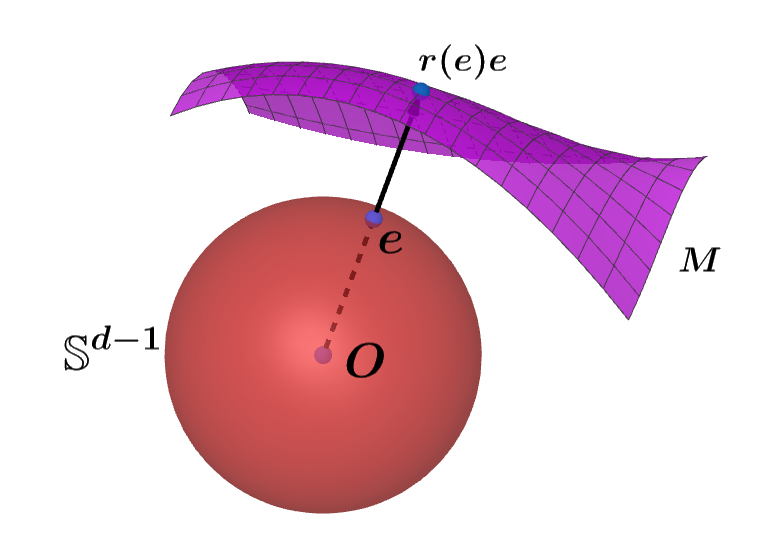}
    \caption{$M$ can be described by $\imath$}
    \label{technical}
\end{figure}
\FloatBarrier

\begin{remark}
In the statement, we require $r \in C^1(K) $, which means that the function $r$ can actually be extended to a small neighborhood of $K \subset \mathbb{S}^{d-1}$ such that the extension is $C^1$.
\end{remark}

\begin{remark}
In fact, the constant $2$ in $2\delta$ can be replaced by any value larger than $1$. The key point is that when applying Theorem \ref{projection theorem} to a compact subset $E \subset M$, where $M$ satisfies the conditions of this lemma, we can always take $l = 0$. Here, $l$ corresponds to $l_F$ in the statement of Theorem \ref{projection theorem}. This follows from the fact that the intersection region is contained within a subtube of the same radius, but with a length $\lesssim \delta$. In the proof, we refer to the length of the shortest such tube as the "intersection length."
\end{remark}

\begin{proof}
Since $r$ is continuous and strictly positive, the hypersurface $M$ is bounded away from the origin. Moreover, let $R$ be the maximum of $r$ on $K$, and let $R_0$ be the minimum. Consider any line with direction $e$ that passes through the origin. Note that the tube $T_\delta(e)$ is the $\delta$-neighborhood of this line. The tube $T_\delta(e)$ only intersects $M$ in the region $\imath(\overline{B_{C R_0^{-1}\delta} (e)} \cap \mathbb{S}^{d-1})$ where $C>0$ is a constant independent of both $e$ and $\delta$. Therefore, we only need to consider the function $r$ on a $C R_0^{-1} \delta$-neighborhood of $e$ in $\mathbb{S}^{d-1}$. Let $y_{\max{}}$ be the maximum of the function
$$
\function{\Pi^e}{\mathbb{R}^d}{\mathbb{R}_{>0}}{x}{x\cdot e}
$$
within the region $\imath(\overline{B_{C R_0^{-1}\delta} (e)} \cap \mathbb{S}^{d-1})$, and let $y_{\min{}}$ be the minimum, attained by $\xi_{\max{}}$ and $\xi_{\min{}}$, respectively. It follows by simple geometry that
$$
\text{intersection length} = y_{\max{}} - y_{\min{}}.
$$
We also let $r_{\max{}}$ be the maximum of the function $r$ within $\overline{B_{C R_0^{-1}\delta} (e)} \cap \mathbb{S}^{d-1}$, and let $r_{\min{}}$ be the minimum. Assume that they are attained by $\eta_{\max{}}$ and $\eta_{\min{}}$, respectively. Therefore, we have
$$
\begin{aligned}
\text{intersection length} 
&= y_{\max{}}-y_{\min{}} = \xi_{\max{}} \cdot e - \xi_{\min{}} \cdot e \\
&\leq |\xi_{\max{}}| - \sqrt{|\xi_{\min{}}|^2-\delta^2}\\
&\leq\; r_{\max{}} - \sqrt{r_{\min{}}^2-\delta^2}
=\; \frac{r_{\max{}}^2 - r_{\min{}}^2 +\delta^2}{r_{\max{}}+\sqrt{r_{\min{}}^2-\delta^2}}
\leq \frac{r_{\max{}}^2 - r_{\min{}}^2 +\delta^2}{R_0+ \frac{1}{2}R_0}\\
&\leq\; \frac{2R(r_{\max{}} - r_{\min{}}) +\delta^2}{\frac{3}{2}R_0}\\
&\leq\; \frac{2R\cdot C_1\delta+ \delta^2}{\frac{3}{2}R_0} \leq\; C_2 \delta.
\end{aligned}
$$
where the intersection length refers to the length of the intersection. The above implies that the intersection is contained in a subtube of radius $\delta$ but length $\lesssim \delta$, which can be covered by $N$ balls of radius $2\delta$ where $N$ is a constant independent of both $\delta$ and $e$.

\end{proof}

\begin{figure}[!ht]
    \centering
    \includegraphics[scale=0.6]{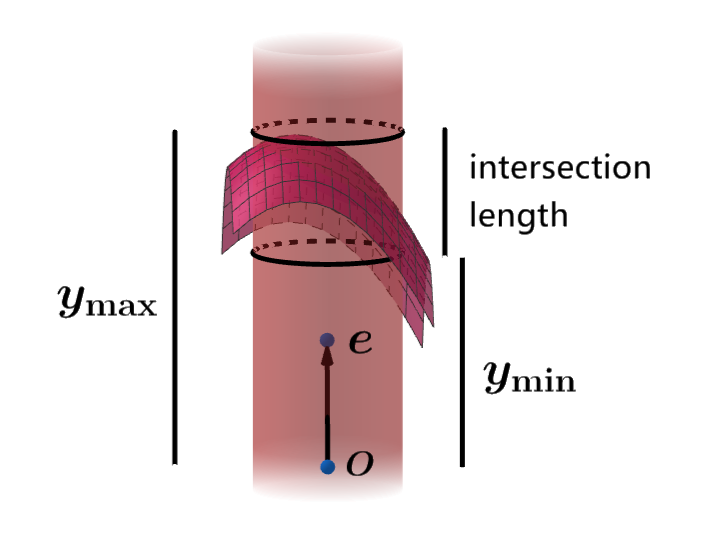}
    \caption{Geometric illustration of the intersection length in Lemma \ref{technical lemma}.}
    \label{intersection length}
\end{figure}
\FloatBarrier

\begin{remark}
It is noteworthy that we do not use the assumption that $M$ is a $(d-1)$-dimensional hypersurface. In fact, the proof still works for smaller compact sets $F$ that can be described using the polar coordinates and a radius function $r \in C^1(\imath^{-1}(F))$. Here, $C^1(\imath^{-1}(F))$ means the function $r$, defined on the compact subset $\imath^{-1}(F) \subset \mathbb{S}^{d-1}$, admits a $C^1$ extension to an open neighborhood of this domain. 
\end{remark}

\section{Proof of the main result}
In this section, we first state our strategy. We then identify two key obstacles and address them separately. Finally, we give a proof of the main results, serving as a concluding argument.

\subsection{Strategy}
We will express the dot product on paraboloids in terms of one lower-dimensional distance, enabling us to apply the pinned distance results. However, two key issues must be addressed.

First, this approach does not work when $|\bar{x}|^2 + a_d = 0$, which is evident from the definition of transformation maps. Second, we must ensure that the transformation maps preserve dimension. By Lemma \ref{not lose dimension} and Proposition \ref{Jacobian of transformation maps}, the dimension is preserved as long as the set under consideration stays away from the region where the Jacobian vanishes.

Finally, we have to address the following two problematic parts:
\begin{enumerate}
    \item Singularity: The region where $|\bar{x}|^2+a_d=0$, which occurs when $a_d \leq 0$.
    \item Degenerate region: The region where $T_a$ has zero Jacobian. By Proposition \ref{Jacobian of transformation maps}, this region is given by
    $$
    \{ (\bar{x}+\bar{a}, |\bar{x}|^2 + a_d) \colon \sum_{i=1}^{d-1}(x_i+a_i)^2 = a_1^2 +\cdots + a_{d-1}^2 +a_d \}
    $$
    where $\bar{x}=(x_1,\cdots,x_{d-1})$. This region occurs when $a_1^2 +\cdots + a_{d-1}^2 +a_d \geq 0$.
\end{enumerate}

\subsection{Singularity}\label{Singularity}
From the calculations in the previous section, it is clear that we cannot analyze the dot product through pinned distance results for compact subsets of the singularity:
$$
S \coloneqq \{ (\bar{x}+\bar{a},|\bar{x}|^2+a_d) \colon |\bar{x}|^2+a_d=0\} = 
\{ (\bar{x}+\bar{a},0) \colon |\bar{x}|^2+a_d=0\}.
$$
Instead, we analyze the dot product directly using Theorem \ref{projection theorem}.

Note that the singularity set
$$
S = \{ (\bar{x}+\bar{a},|\bar{x}|^2+a_d) \colon |\bar{x}|^2+a_d=0\} = 
\{ (\bar{x}+\bar{a},0) \colon |\bar{x}|^2+a_d=0\}
$$
occurs only when $a_d\leq 0$. If $a_d=0$, then $S$ consists of a single point, which together with a small neighborhood of itself can be removed by the pigeonhole principle. From now on, assume throughout this subsection that $a_d<0$.

When considering the dot product, the $d$-th component is always zero, so we can focus only on the first $d-1$ components and work within $\mathbb{R}^{d-1}$. Note again that
$$
\{ \bar{x}+\bar{a} \colon |\bar{x}|^2+a_d=0\}
$$
is a $(d-2)$-dimensional sphere of radius $\sqrt{-a_d}$ centered at $\bar{a}$.

\begin{figure}[!ht]
    \centering
    \includegraphics[scale=0.5]{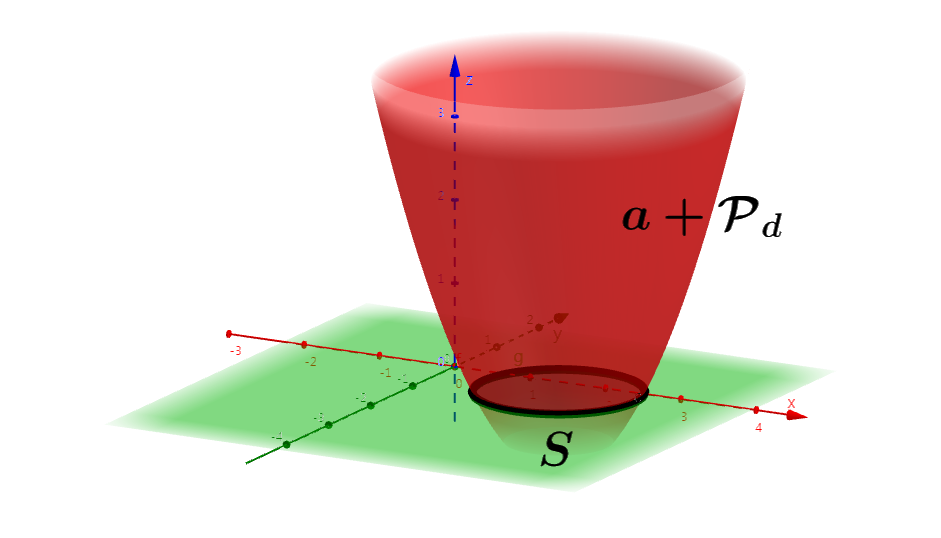}
    \caption{Singularity}
    \label{singularity}
\end{figure}
\FloatBarrier

There are three cases to distinguish: whether the origin is inside the sphere, on the sphere, or outside the sphere. We explore the dimensional threshold ensuring every compact subset $E$ of $\sqrt{-a_d}\;\mathbb{S}^{d-2}+\bar{a}$ has a product set of positive Lebesgue measure. (Note: Since we are working in $\mathbb{R}^{d-1}$, we need to replace $d$ with $d-1$ when applying Theorem \ref{projection theorem}.) 
~\\

\noindent \textbf{\underline{Case 1: $O$ is inside $\sqrt{-a_d}\;\mathbb{S}^{d-2}+\bar{a}$}}
~\\\\
In this case, we set $K = \mathbb{S}^{d-2}$ and $M = \sqrt{-a_d}\;\mathbb{S}^{d-2}+\bar{a}$. Applying Lemma \ref{technical lemma} in $(d-1)$ dimensions, we find that for any tube $T_\delta(e)$, which has length $\approx 1$ and radius $\delta$, emanating from the origin in the direction $e$, the intersection with $M$ can be covered by at most $N$ balls of radius $2\delta$. 

Thus, for any $s$-dimensional Frostman measure $\mu$ supported on the sphere $\sqrt{-a_d}\;\mathbb{S}^{d-2}+\bar{a}$, we obtain the bound:
$$
\mu(T_\delta(e)) \leq \sum_{i=1}^N \mu(B_i) \lesssim N (2\delta)^s = (N 2^s) \delta^s,
$$
where $B_1,\dots,B_N$ are balls of radius $2\delta$ covering the intersection of $T_\delta(e)$ and $M$.

Finally, we apply Theorem \ref{projection theorem} and conclude that the dimension threshold in this case is $\frac{d-1+0}{2}=\frac{d-1}{2}$.

\begin{figure}[!ht]
    \centering
    \includegraphics[scale=0.5]{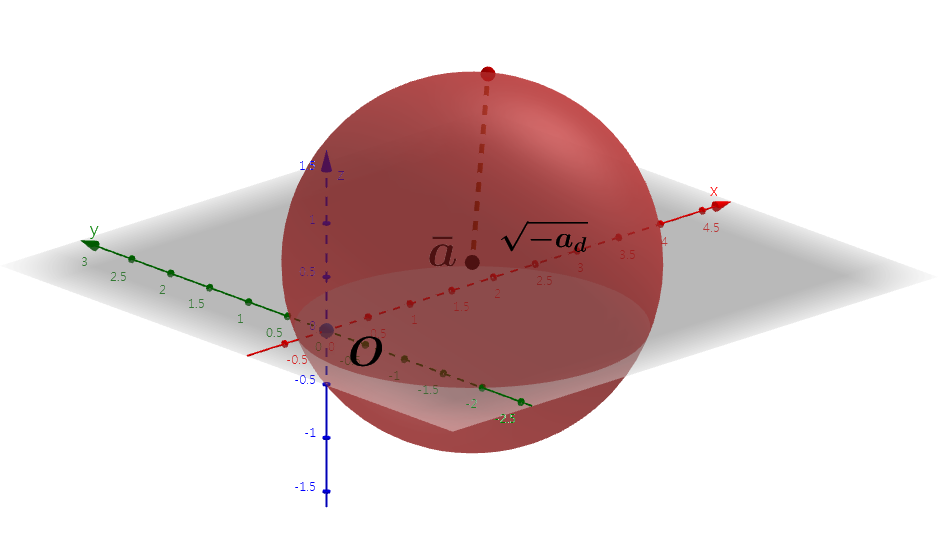}
    \caption{$O$ is inside the sphere.}
    \label{inside the sphere}
\end{figure}
\FloatBarrier

~\\

\noindent \textbf{\underline{Case 2: $O$ lies on $\sqrt{-a_d}\;\mathbb{S}^{d-2}+\bar{a}$}}
~\\\\
In this case, we exclude the origin and a sufficiently small $\varepsilon$-neighborhood around it. We define $K$ as the half-sphere with a small neighborhood of the boundary removed, and set $M = (\sqrt{-a_d}\;\mathbb{S}^{d-2}+\bar{a}) \setminus B_\varepsilon(O)$. We then apply Lemma \ref{technical lemma}, since the function $r > 0$ can be extended to a $C^1$ function on the half-sphere excluding the boundary. Finally, applying Theorem \ref{projection theorem}, we obtain the same dimensional threshold $\frac{d-1}{2}$.

\begin{figure}[!ht]
    \centering
    \includegraphics[scale=0.5]{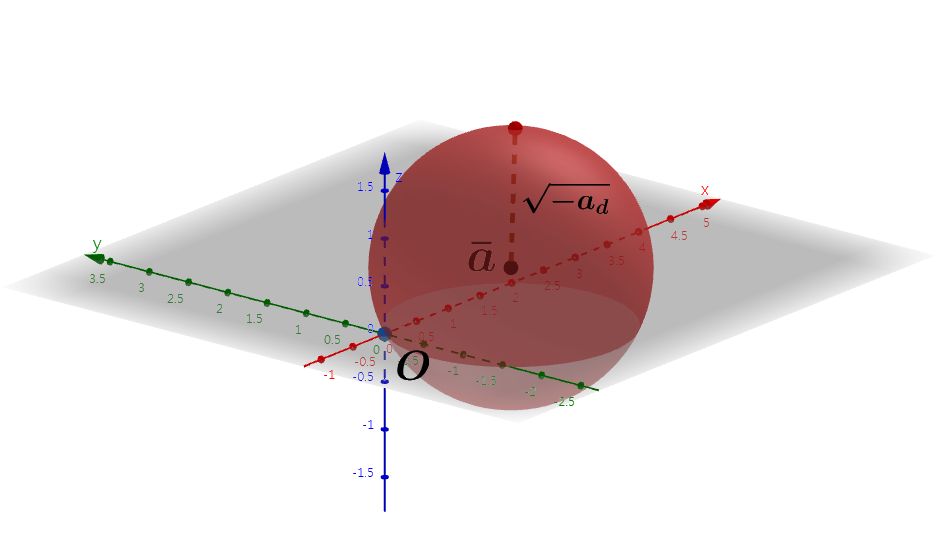}
    \caption{$O$ lies on the sphere.}
    \label{on the sphere}
\end{figure}
\FloatBarrier

~\\

\noindent \textbf{\underline{Case 3: $O$ is outside $\sqrt{-a_d}\;\mathbb{S}^{d-2}+\bar{a}$}}
~\\\\
This case is more complicated, requiring us to divide the translated sphere into three parts: the visible part, the tangent part, and the invisible part. We apply Lemma \ref{technical lemma} to both the visible and invisible parts, excluding a small neighborhood around the tangent part, which yields the threshold $\frac{d-1}{2}$. The tangent part, which is a $(d-3)$-dimensional sphere, along with its small neighborhood represents the most challenging scenario.

\begin{figure}[!ht]
    \centering
    \includegraphics[scale=0.5]{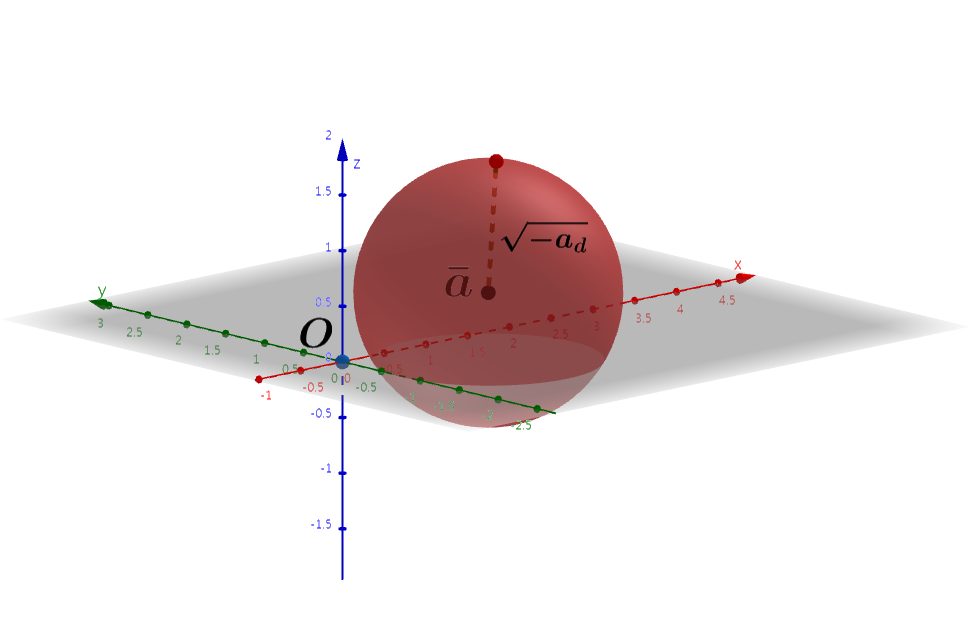}
    \caption{$O$ is outside the sphere.}
    \label{outside the sphere}
\end{figure}
\FloatBarrier

\begin{figure}[!ht]
    \centering
    \includegraphics[scale=0.5]{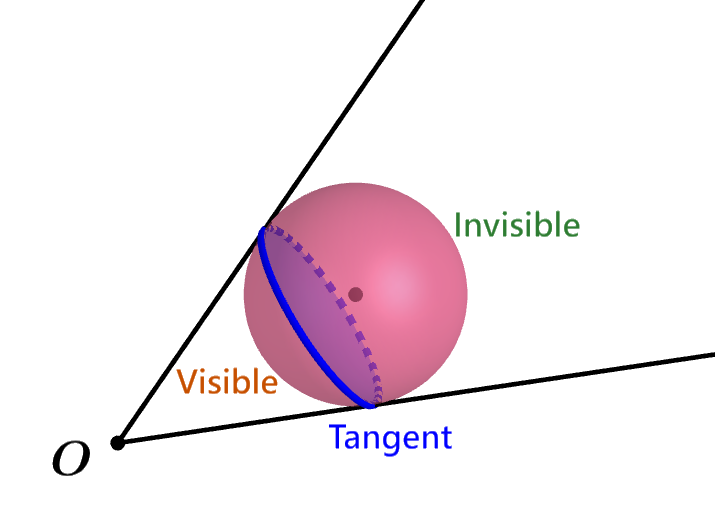}
    \caption{Visible, tangent, and invisible parts of the sphere}
    \label{divided sphere}
\end{figure}
\FloatBarrier

Now, we show that the set under consideration, $E$, which is a compact subset of the sphere $S = \sqrt{-a_d}\; \mathbb{S}^{d-2} + \bar{a}$, can be reduced either to a compact subset of the tangent part $T$, or to a compact subset that does not intersect $T$. This is done via a pigeonholing argument.

Suppose $\dim_H E > \alpha$ for a dimensional threshold $\alpha \in (0, d)$. If $\dim_H (E \cap T) > \alpha$, then we consider the compact set $E \cap T$, which is a compact subset of both $E$ and $T$ with dimension greater than $\alpha$. On the other hand, if $\dim_H (E \cap T) \leq \alpha$, pick some $0 < \varepsilon < \dim_H E - \alpha$. Denote by $\mathcal{H}^{\alpha+\varepsilon}$ the $(\alpha+\varepsilon)$-dimensional Hausdorff measure, and let $T^{\beta}$ be the $\beta$-neighborhood of $T$. By the fact that Hausdorff measures are Borel and by monotonicity,
$$
\infty = \mathcal{H}^{\alpha+\varepsilon}(E \ \setminus \ T) = \lim_{n \to \infty} \mathcal{H}^{\alpha+\varepsilon}(E \ \setminus \ T^{\frac{1}{n}}).
$$
Thus, there exists $N$ large enough such that
$$
\mathcal{H}^{\alpha+\varepsilon}(E \ \setminus \ T^{\frac{1}{N}}) > 0,
$$
which implies that the set $E \ \setminus \ T^{\frac{1}{N}}$ has dimension at least $\alpha + \varepsilon$. Moreover, it is a compact set bounded away from $T$. Therefore, we have reduced $E$ to a compact subset that does not intersect $T$ and has dimension greater than $\alpha$.

We now handle the case where $E$ is a compact subset of $S$ that does not intersect $T$. Note that in this case, either the portion in the visible part or that in the invisible part has large dimension. Recall that both the visible and invisible parts can be parametrized by polar coordinates. Therefore, the dimensional threshold ensuring that $E$ has a dot product set of positive measure is $\frac{d-1}{2}$ in this case.

The case where $E$ is a compact subset of $T$ requires closer examination. Observe that all the points in $T$ have the same distance to the origin. When considering the projection $\imath^{-1}(T)$ of $T$ onto $\mathbb{S}^{d-2}$, the radius function $r \colon \imath^{-1}(T) \to \mathbb{R}_{>0}$ is constant on its domain. This shows that it admits a $C^1$ extension to an open neighborhood of its domain, implying $r \in C^1(\imath^{-1}(T))$. Recall that the remark following Lemma \ref{technical lemma} relaxes the assumption and enables the lemma to apply to smaller sets. Therefore, we are able to apply Lemma \ref{technical lemma} to $T$ and $\imath^{-1}(T)$ and conclude that the intersection of $T$ with each tube $T_\delta(e)$ has length $\lesssim \delta$. Hence, the exponent is also $\frac{d-1}{2}$.

In fact, we can avoid using the remark after Lemma \ref{technical lemma}. Note that $T$ is contained in $R\; \mathbb{S}^{d-2}$, where $R = \sqrt{|\bar{a}|^2 + a_d}$ is the distance from each point in $T$ to the origin. We can then apply Lemma \ref{technical lemma} directly to $R\; \mathbb{S}^{d-2}$. Thus, since $(T_\delta(e) \cap T) \subset (T_\delta(e) \cap R\; \mathbb{S}^{d-2})$, it can be covered by at most $N$ balls of radius $2\delta$, which yields the exponent $\frac{d-1}{2}$. 

Although the two arguments appear different, they share a common insight: embed $T$ into another hypersurface that can be described using polar coordinates.

By discussing the three cases, we also obtain a by-product theorem. (Note: When discussing the singularity $S$, we are working within $\mathbb{R}^{d-1}$. However, the underlying space of the following theorem is $\mathbb{R}^d$. As a result, the exponents may appear different.)

\begin{theorem}
Let $a \in \mathbb{R}^d$ be any translation vector. Suppose $E$ is a compact subset of the translated sphere $a + \mathbb{S}^{d-1}$. If $\dim_H E > \frac{d}{2}$, then there exists $x \in E$ such that
$$
|\Pi^x(E)| > 0.
$$
In particular, $|\Pi(E)| > 0$.
\end{theorem}

\begin{remark}
This generalizes Corollary \ref{Sphere dot product}. Consider a translated sphere that does not enclose the origin, and consider a tangent tube. The intersection of such a tube and the sphere has length $\approx \delta^{\frac{1}{2}}$, which requires $\approx \delta^{-\frac{1}{2}}$ balls of radius $2\delta$ to cover it. This is because the intersection lies within a spherical cap of thickness $\delta$. Hence, I previously derived the exponent $\frac{d + 1/2}{2} = \frac{d}{2} + \frac{1}{4}$. Later, I realized that a sharper exponent $\frac{d}{2}$ can be obtained by dividing the sphere, reducing to a subset via the pigeonhole principle, and embedding the tangent part into another hypersurface.
\end{remark}

\begin{figure}[!ht]
    \centering
    \includegraphics[scale=0.4]{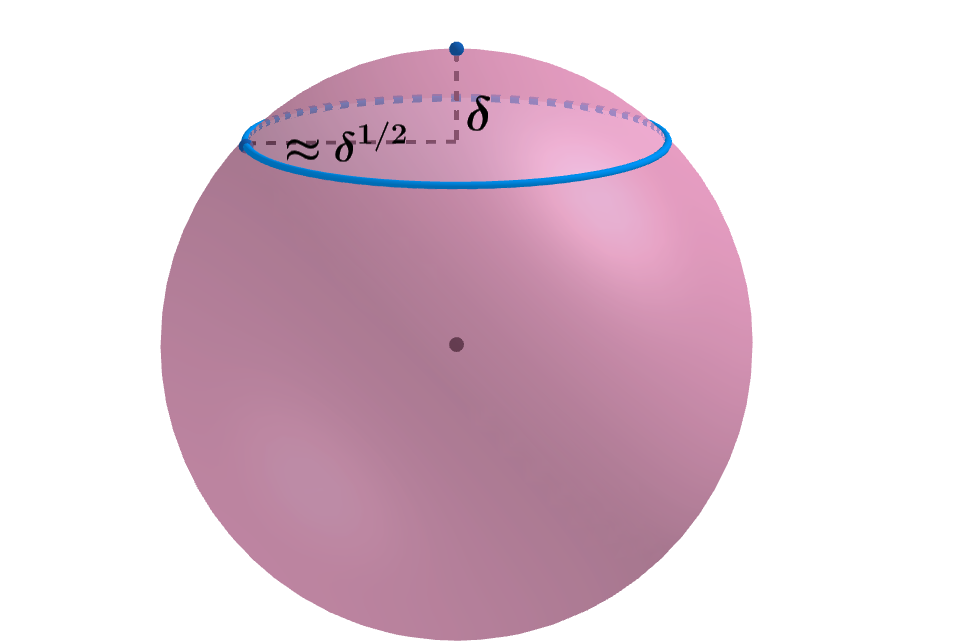}
    \caption{Spherical cap of thickness $\delta$}
    \label{spherical cap}
\end{figure}
\FloatBarrier

\subsection{Degenerate region}\label{Degenerate region}
Recall that the degenerate region refers to the set of points in $\mathbb{R}^{d-1}$ where the transformation map has zero Jacobian. From the calculations in Proposition \ref{Jacobian of transformation maps}, this set is given by:
$$
\bar{D} \coloneqq \{ \bar{x}=(x_1,\dots, x_{d-1}) \colon \sum_{i=1}^{d-1}(x_i+a_i)^2 = a_1^2 +\cdots +a_{d-1}^2 + a_d \, \}.
$$
If $a_1^2 +\cdots +a_{d-1}^2 + a_d<0$, then the set is empty. If $a_1^2 +\cdots +a_{d-1}^2 + a_d=0$, then the set consists of a single point, which also corresponds to a single point on the paraboloid $a+\mathcal{P}_d$ and can be ignored using the pigeonhole principle. From now on, we assume that $a_1^2 +\cdots +a_{d-1}^2 + a_d>0$ throughout this subsection.

We now examine the corresponding degenerate region $D$ on the translated paraboloid $a+\mathcal{P}_d$:
$$
\begin{aligned}
D &= \{ (\bar{x}+\bar{a},|\bar{x}|^2+a_d) \colon \bar{x} \in \bar{D} \}\\
&= \{ (\bar{x}+\bar{a},|\bar{x}|^2+a_d) \colon \sum_{i=1}^{d-1}(x_i+a_i)^2 = a_1^2 +\cdots +a_{d-1}^2 + a_d \, \}\\
&= \{ (\bar{x}+\bar{a},|\bar{x}|^2+a_d) \colon |\bar{x}+\bar{a}|^2 = a_1^2 +\cdots +a_{d-1}^2 + a_d \, \}.
\end{aligned}
$$
It follows that the degenerate region $D$ on the paraboloid is precisely the intersection of the translated paraboloid and the cylinder:
$$
\{ (y_1,\dots,y_{d-1},y_d) \in \mathbb{R}^d \colon \sum_{i=1}^{d-1} y_i^2 = a_1^2 +\cdots +a_{d-1}^2 + a_d \, \}.
$$
In other words, the degenerate region $D$ can be characterized by the following two algebraic equations:
$$
\left\{
\begin{aligned}
    \, & y_d -a_d = (y_1-a_1)^2 + \cdots + (y_{d-1}-a_{d-1})^2, \\ 
    \, & y_1^2 + \cdots + y_{d-1}^2 = a_1^2 +\cdots +a_{d-1}^2 + a_d.
\end{aligned}
\right.
$$
Substituting the second equation into the first, we derive the equation of a hyperplane $H$:
$$
y_d = 2 (a_1^2 + \cdots + a_{d-1}^2 + a_d ) - 2a_1 y_1 - \cdots - 2 a_{d-1} y_{d-1}.
$$
This defines a $(d-1)$-dimensional hyperplane in $\mathbb{R}^d$ that does not contain the origin. Since the hyperplane does not pass through the origin, it is bounded away from it. By basic high-dimensional geometry, we conclude that the degenerate region $D$, which is the intersection of the hyperplane $H$ and the cylinder, forms a $(d-2)$-dimensional ellipsoidal hypersurface.

\begin{figure}[!ht]
    \centering
    \includegraphics[scale=0.5]{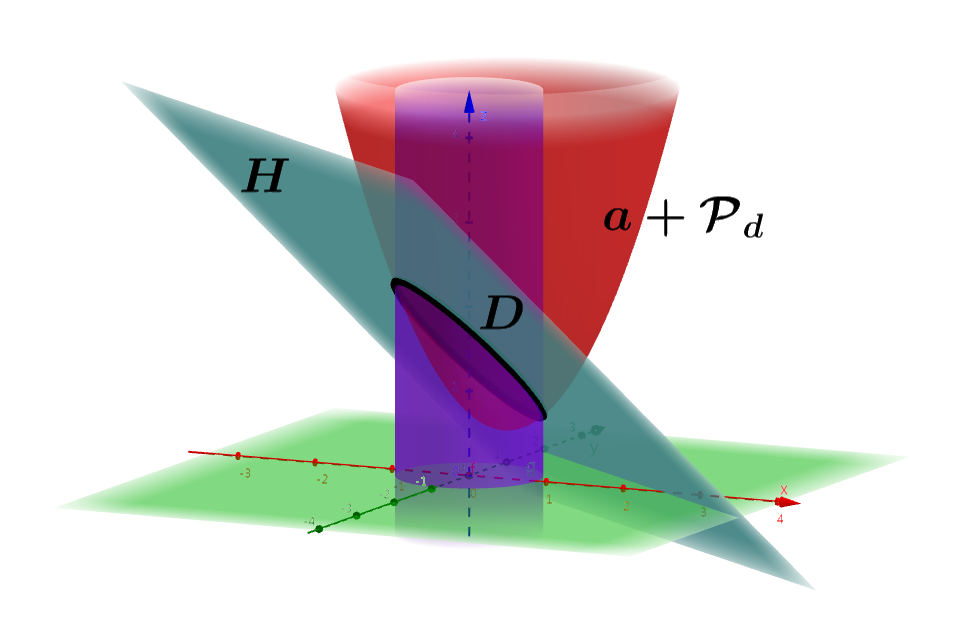}
    \caption{Degenerate region}
    \label{degenerate region}
\end{figure}
\FloatBarrier

Now, we make a crucial observation. We can apply a suitable rotation $g \in \mathbb{O}(d)$ to the hyperplane $H$ to bring it into the following form, while preserving the dot product since $g\in \mathbb{O}(d)$:
$$
gH=\left\{ (z_1,\dots,z_{d-1}, z_d ) \colon z_d= \frac{2(a_1^2+\cdots+a_{d-1}^2+a_d)}{\sqrt{4a_1^2 + \cdots + 4a_{d-1}^2 +1 }} \right\}.
$$

That is, the first $d-1$ coordinates are free, while the last coordinate represents the distance between the origin and $H$. Since $g$ is a rigid motion, the degenerate region remains an ellipsoidal hypersurface on the $(d-1)$-dimensional hyperplane $gH$. Therefore, as in the discussion of singularity, we only need to focus on the dot product in the first $(d-1)$ coordinates, as the last coordinate is now a constant. This allows us to lower the dimensional threshold.

Now, we work in $\mathbb{R}^{d-1}$ and again apply Theorem \ref{projection theorem} to analyze the dot product. Specifically, we must consider whether the origin $O \in \mathbb{R}^{d-1}$ lies inside, outside, or on the ellipsoidal hypersurface $gD$, as discussed in Subsection \ref{Singularity}. Recall that an ellipsoid can be transformed into a ball via a rescaling linear isomorphism. Hence, the arguments in Subsection \ref{Singularity} still apply to $gD$ with slight modifications. To conclude, all cases yield the same exponent $\frac{d-1}{2}$. Therefore, if $E$ is a compact subset of the degenerate region $D$ and $\dim_H E > \frac{d-1}{2}$, then there exists $x \in E$ such that $|\Pi^x(E)| > 0$.

\subsection{Proof of our main theorem for $d=3$}
The proof for the $d=3$ case is relatively straightforward because both the singularity and the degenerate region are at most $d-2=1$-dimensional, which is below the threshold $\frac{5}{4}$ derived from the pinned distance result.

We now proceed with the proof of the main theorem for $d=3$. Fix $a=(\bar{a},a_3)$. For any compact subset $E$ of a translated paraboloid $a+\mathcal{P}_3 = (\bar{a},a_3)+\mathcal{P}_3$ with Hausdorff dimension greater than $\frac{5}{4}$, we define
$$
\bar{E} = \{ \bar{x} \colon   (\bar{x}+\bar{a},|\bar{x}|^2+a_3) \in E \}.
$$
Observe that there is a one-to-one correspondence between $\bar{E} \subset \mathbb{R}^{2}$ and $E\subset a+\mathcal{P}_3 \subset \mathbb{R}^3$. Since both the singularity and the degenerate region are at most $1$-dimensional, they can be removed along with a small neighborhood surrounding them without losing too much dimension by the pigeonhole principle. More precisely, we can obtain a compact subset of $E$ that still has Hausdorff dimension greater than $\frac{5}{4}$ while being bounded away from the singular and degenerate regions. Thus, without loss of generality, we may assume $E$ is bounded away from these bad parts.

Fix $a=(\bar{a},a_3)$ and the compact set $E$ on $a+\mathcal{P}_3$ with dimension greater than $\frac{5}{4}$ and bounded away from the singular and degenerate regions. By the key observation in Section \ref{key observation}, for each $(\bar{x},|\bar{x}|^2)+a$, we have
$$
|\Pi^{(\bar{x},|\bar{x}|^2)+a} (E)|>0 \;\;\;\; \Longleftrightarrow \;\;\;\;  |\Delta^{-\frac{\bar{x}+\bar{a}}{2(|\bar{x}|^2+a_3)}}
(\bar{E})| = |\Delta^{T_a(\bar{x})}(\bar{E})| >0.
$$
Since $\bar{E}$ is bounded away from the singularity, $T_a(\bar{x})$ is well-defined for all $\bar{x}\in \bar{E}$, and since it is also bounded away from the degenerate region, Proposition \ref{Jacobian of transformation maps} and Lemma \ref{not lose dimension} ensure that $\dim_H (T_a(\bar{E}))=\dim_H (\bar{E}) =\dim_H E > \frac{5}{4}$.

The crucial step is to apply Theorem \ref{state-of-the-art pinned distance we use} for dimension $d-1=3-1=2$ to the sets $E_1=\bar{E}$ and $E_2=T_a(\bar{E})$. This guarantees the existence of some $\bar{y}=T_a(\bar{x}) \in T_a(\bar{E})$ where $\bar{x} \in \bar{E}$ such that
$$
|\Delta^{\bar{y}}(\bar{E})|=|\Delta^{T_a(\bar{x})}(\bar{E})| >0.
$$
Consequently, we obtain
$$
|\Pi^{(\bar{x},|\bar{x}|^2)+a} (E)|>0,
$$
which implies $|\Pi(E)|>0$, thereby completing the proof for the case $d=3$. \qed

\subsection{Proof of our main theorem for $d\geq 4$}

Let $a = (\bar{a}, a_d)$ be an arbitrary translation vector in $\mathbb{R}^d$, and let $E$ be a compact subset of $a + \mathcal{P}_d$ with $\dim_H E > \frac{d}{2} - \frac{1}{4} - \frac{1}{8d - 4}$. By a pigeonholing argument, we can reduce to three cases: 

\begin{itemize}
    \item $E$ is a compact subset of $S$ with $\dim_H E > \frac{d}{2} - \frac{1}{4} - \frac{1}{8d - 4}$,
    \item $E$ is a compact subset of $D$ with $\dim_H E > \frac{d}{2} - \frac{1}{4} - \frac{1}{8d - 4}$, or
    \item $E$ does not intersect either $S$ or $D$ and still satisfies $\dim_H E > \frac{d}{2} - \frac{1}{4} - \frac{1}{8d - 4}$.
\end{itemize}

If $E$ is a compact subset of $S$ with $\dim_H E > \frac{d}{2} - \frac{1}{4} - \frac{1}{8d - 4} > \frac{d - 1}{2}$, then by the discussion in Subsection \ref{Singularity}, there exists $x \in E$ such that $|\Pi^x(E)| > 0$.

If $E$ is a compact subset of $D$ with $\dim_H E > \frac{d}{2} - \frac{1}{4} - \frac{1}{8d - 4} > \frac{(d - 1)}{2}$, then by the discussion in Subsection \ref{Degenerate region}, there exists $x \in E$ such that $|\Pi^x(E)| > 0$.

Finally, suppose $E$ is a compact set that does not intersect $S$ or $D$. First its projection $\bar{E} \subset \mathbb{R}^{d-1}$ satisfies
$$
\dim_H \bar{E} = \dim_H E > \frac{d}{2} - \frac{1}{4} - \frac{1}{8d - 4} 
= \frac{(d - 1)}{2} + \frac{1}{4} - \frac{1}{8(d - 1) + 4}.
$$
Since $E$ does not intersect $S$, $T_a$ is well-defined on $\bar{E}$.
Moreover, since $E$ does not intersect $D$, the Jacobian determinant of $T_a$ does not vanish on $\bar{E}$ or on a small open neighborhood of it. By Proposition \ref{not lose dimension},
$$
\dim_H T_a(\bar{E}) = \dim_H \bar{E} > \frac{d}{2} - \frac{1}{4} - \frac{1}{8d - 4} 
= \frac{(d - 1)}{2} + \frac{1}{4} - \frac{1}{8(d - 1) + 4}.
$$

Now, apply Theorem \ref{state-of-the-art pinned distance we use} for dimension $d - 1 \geq 3$ to $\bar{E}$ and $T_a(\bar{E})$. We obtain some $\bar{y} = T_a(\bar{x}) \in T_a(\bar{E})$, where $\bar{x} \in \bar{E}$, such that
$$
|\Delta^{\bar{y}}(\bar{E})|=|\Delta^{T_a(\bar{x})}(\bar{E})| > 0.
$$
By the key observation \eqref{positive measure iff} in Subsection \ref{key observation},
$$
|\Delta^{-\frac{\bar{x} + \bar{a}}{2(|\bar{x}|^2 + a_d)}}(\bar{E})| > 0 
\;\;\;\; \implies \;\;\;\;
|\Pi^{(\bar{x}, |\bar{x}|^2) + a}(E)| > 0.
$$
The proof for $d \geq 4$ is complete. \qed

\section{Sharpness of our result}

\subsection{Why the case $d=2$ is not considered}\label{Not study d=2}
We do not study the case $d=2$ because the method we use would lead us to work with the pinned distance problem in $\mathbb{R}^{d-1} = \mathbb{R}^1$. However, there exists a counterexample of a compact set $E \subset \mathbb{R}^1$ with Hausdorff dimension 1, whose distance set has zero Lebesgue measure.

Moreover, we now show that there exists a compact subset $E \subset \mathcal{P}_2 = \{ (x,x^2) \colon x\in \mathbb{R} \}$ with $\dim_H E = 1$ but $|\Pi(E)|=0$. First, observe that for $(x,x^2),(y,y^2) \in \mathcal{P}_2$, where $x,y \in \mathbb{R}$,
$$
(x,x^2) \cdot (y,y^2) = xy + x^2 y^2 = \left(xy+\frac{1}{2}\right)^2 -\frac{1}{4}.
$$
Let $F$ be a compact subset of $\mathbb{R}$ with $\dim_H F = 1$ but $|\Pi(F)|=0$ (such an example indeed exists). We define $E$ as follows:
$$
E \coloneqq \{ (x,x^2) \colon x\in F \}.
$$
Note that $\dim_H E = \dim_H F =1$, and since $\Pi(F)$ has zero Lebesgue measure, so does $\Pi(E)$. Since $\dim_H \mathcal{P}_2 =1$, this is the largest possible threshold. However, this is still insufficient to guarantee that compact subsets of $\mathcal{P}_2$ have dot product sets of positive Lebesgue measure.

\begin{remark}
Although we can construct a counterexample on the standard parabola $\mathcal{P}_2$, we are uncertain whether counterexamples exist on other translated parabolas $(a_1,a_2) + \mathcal{P}_2$. By "counterexamples," we mean compact subsets of $(a_1,a_2)+\mathcal{P}_2$ with Hausdorff dimension 1 whose dot product set has zero Lebesgue measure.
\end{remark}

\subsection{A naïve counterexample on $\mathcal{P}_d$}
In this subsection, we construct an example of a compact subset of the standard paraboloid $\mathcal{P}_d$. When constructing counterexamples of sets that have large dimension but also possess a distance (resp. dot product) set of measure zero, there are very few available methods. The most classical approach is to consider lattice points and their approximations, as lattice points exhibit many repeated distances and dot products. Note that the first $d-1$ coordinates determine the last coordinate on $\mathcal{P}_d$. Therefore, it is natural to consider lattice points in $\mathbb{R}^{d-1}$ and project the set onto $\mathcal{P}_d \subset \mathbb{R}^d$. However, this approach only yields the exponent $\frac{d-1}{4}$ as we shall see below.

Fix $s\in (0,1/2)$. Let $\{q_k\}_{k=1}^\infty$ be a sequence that satisfies $q_{k+1}> q_k^k$. For each $k\in \mathbb{N}$, define
$$
E_{s,q_k} \coloneqq 
\left\{ (x_1,\dots,x_{d-1}) \in [0,1]^{d-1} \colon \exists n_i \in \mathbb{Z} \text{ such that} \left| x_i-\frac{n_i}{q_k}\right| \leq q_k^{-1/s},\, \forall 1\leq i \leq d-1 \right\}
$$
and set $E_s=\bigcap_{k\in \mathbb{N}} E_{s,q_k}$. Follow the argument in Theorem 8.15 in Falconer's book \cite{Falconer's_book}, one can verify that $\dim_H E= s(d-1)$. We define a compact subset $F_s$ of $\mathcal{P}_d$ by
$$
F_s\coloneqq \{ (\bar{x},|\bar{x}|^2) \colon \bar{x}\in E_s \}.
$$
Clearly, we have $\dim_H F_s=s(d-1)$. Now, by appropriately controlling $s$, we can ensure that $|\Pi(F_s)|=0$. The key question is: how large can $s$ be while still ensuring $|\Pi(F_s)|=0$?

Through direct calculations, for any two points $(\bar{x},|\bar{x}|^2)$ and $(\bar{y},|\bar{y}|^2) \in F_s$, where $\bar{x},\bar{y}\in E_s$, their dot product lies in a $C_1 q_k^{-1/s}$-neighborhood of a lattice point of the form $\frac{n}{q_k^4}$, where $C_1$ is independent of $k$, and $n \in \mathbb{Z}$ is within an admissible range since both $E_s$ and $F_s$ are bounded. This implies that $\Pi(F_s)$ is contained in a union of $\leq C_2 q_k^4$ intervals of length $ \leq C_1 q_k^{-1/s}$. Therefore, we conclude that when $s<1/4$, $\Pi(F_s)$ has measure zero. 

To summarize, for every $\varepsilon>0$, we can find a compact subset $F$ of $\mathcal{P}_d$ such that
$$
\dim_H F > \frac{d-1}{4} - \varepsilon , \text{ but } |\Pi(F)|=0.
$$

\begin{remark}
Compared with Proposition \ref{Dot product counterexample}, which provides a counterexample for dot products in $\mathbb{R}^d$, we can see that this construction yields a suboptimal exponent due to the square in the last coordinate, which arises from the algebraic formula of the paraboloid. Without the square in the last coordinate, we would only need to consider lattice points of the form $\frac{n}{q_k^2}$, where $n\in \mathbb{Z}$. Consequently, we would only require $\leq C_2 q_k^2$ intervals of length $\leq C_1 q_k^{-1/s}$ to cover the dot product set (where $C_2$ is also independent of $k$). This would suffice for $s<1/2$, leading to a better exponent.
\end{remark}

\subsection{Another counterexample on $\mathcal{P}_d$}
In this subsection, we assume $d\geq 5$ and use an alternative method to construct a compact subset $F \subset \mathcal{P}_d$ of dimension $\frac{d-3}{2} - \varepsilon$ whose dot product set has measure zero. Roughly speaking, we consider the set
$$
\{ (\bar{x},|\bar{x}|^2) \colon |\bar{x}|^2=1 \}
= \{ (\bar{x},1) \colon |\bar{x}|^2=1 \}.
$$
When analyzing the size of the dot product set in $\mathbb{R}^{d}$, we may ignore the influence of the last coordinate. Focusing on the first $d-1$ coordinates, we observe that this set forms a $(d-2)$-dimensional sphere centered at the origin in $\mathbb{R}^{d-1}$. Furthermore, the dot product on the sphere is, in a sense, equivalent to the distance, which is translation-invariant.

To proceed, we will make use of the following intersection theorem (which can be found in Chapter 7.2 of Mattila's book \cite{Fourier_analysis_and_Hausdorff_dimension}) to guarantee the existence of a compact subset whose dot product set (distance set) has measure zero.

\begin{theorem} \label{intersection theorem}
    Suppose $0<s<d,\, 0<t<d,\, s+t>d,$ and $t>(d+1)/2$. If $A,B$ are Borel subsets of $\mathbb{R}^d$ with $\mathcal{H}^s(A)>0$ and $\mathcal{H}^t(B)>0$, then for $\theta_d$-almost every $g\in \mathbb{O}(d)$,
    $$
    \mathcal{L}^d (\{ z\in \mathbb{R}^d \colon \dim_H A\cap (\tau_z \circ g)(B) \geq s+t-d \})>0,
    $$
    where $\theta_d$ is the Haar measure on $\mathbb{O}(d)$ and $\tau_z$ is the translation map by the vector $z\in \mathbb{R}^d$.
\end{theorem}

\begin{remark}
The reason for assuming $d\geq 5$ is that we let $B=\mathbb{S}^{d-2}\subset \mathbb{R}^{d-1}$ and $t = d-2$. On the other hand, the theorem also requires that $t>((d-1)+1)/2$. Combining these two conditions, we obtain $d>4$, which is equivalent to $d\geq 5$ since $d\in \mathbb{N}$.
\end{remark}

Now, we construct the desired counterexample. For every sufficiently small $\varepsilon>0$, we can construct a compact subset $E$ of $\mathbb{R}^{d-1}$ such that $\dim_H E > \frac{d-1}{2}-\varepsilon$ and $|\Delta(E)|=0$. By setting the dimension as $d-1$, choosing $A=E,\, B=\mathbb{S}^{d-2},\, s=\frac{d-1}{2}-\varepsilon,$ and $t=d-2$ ($s+t>(d-1)$ is satisfied), we apply Theorem \ref{intersection theorem} and obtain that for $\theta_{d-1}$-almost every $g\in \mathbb{O}(d-1)$,
$$
\mathcal{L}^{d-1} \left( \left\{ z\in \mathbb{R}^{d-1} \colon \dim_H E\cap (\tau_z \circ g)(\mathbb{S}^{d-2}) \geq \frac{d-1}{2}-\varepsilon-1 \right\} \right)>0.
$$

Observe that $g \in \mathbb{O}(d-1)$ preserves the sphere and that the Hausdorff dimension is translation-invariant. Hence, we conclude that there exists some $z\in \mathbb{R}^{d-1}$ such that
$$
\dim_H (\tau_{-z}(E) \cap \mathbb{S}^{d-2}) \geq \frac{d-3}{2} - \varepsilon.
$$
Taking $\bar{F}= \tau_{-z}(E) \cap \mathbb{S}^{d-2}$, we see that $\bar{F}$ is a compact subset of $\mathbb{S}^{d-2}\subset \mathbb{R}^{d-1}$ with dimension at least $\frac{d-3}{2}-\varepsilon$, and since its distance set satisfies $|\Delta(\bar{F})| \leq |\Delta(\tau_{-z}(E))| = |\Delta(E)|=0$, it follows that $|\Pi(\bar{F})|=0$. Finally, defining 
$$
F\coloneqq \{ (\bar{x},1) \mid \bar{x}\in \bar{F} \subset \mathbb{S}^{d-2}\subset \mathbb{R}^{d-1} \},
$$
we obtain a compact subset of $\mathcal{P}_d$ with dimension at least $\frac{d-3}{2}-\varepsilon$ and a dot product set of measure zero.

\begin{remark}
In the previous subsection, we constructed a compact subset of $\mathcal{P}_d$ with dimension at least $\frac{d-1}{4}-\varepsilon$ and a dot product set of measure zero. In this subsection, we show that there ``exists" a compact subset of $\mathcal{P}_d$ with dimension at least $\frac{d-3}{2}-\varepsilon$ and a dot product set of measure zero. It appears that the latter provides a considerably larger exponent than the former for large $d$. However, it is important to contrast that the former is an explicit construction, whereas we only establish the ``existence" of the latter.
\end{remark}

\subsection{Counterexamples on translated paraboloids}
In the previous two subsections, we constructed counterexamples on the standard paraboloid $\mathcal{P}_d$. In this subsection, we assume $d\geq 4$ and consider translations of $\mathcal{P}_d$ and demonstrate the existence of a set $J\subset \mathbb{R}^d$ with $\mathcal{L}^d(J)>0$ such that for each $a\in J$, there exists a compact subset $F_a$ of $a+\mathcal{P}_d$ with dimension at least $\frac{d-2}{2}-\varepsilon$ and a dot product set of measure zero.

The key insight is the intersection theorem (Theorem \ref{intersection theorem}). Again, let $E$ be a compact subset of $\mathbb{R}^{d}$ such that $\dim_H E > \frac{d}{2}-\varepsilon$ and $|\Pi(E)|=0$. By setting $A=E,\,B=\mathcal{P}_d,\, s=\frac{d}{2}-\varepsilon,$ and $t= d-1$ (both the conditions $s+t>d$ and $t> (d+1)/2$ are satisfied) and by applying Theorem \ref{intersection theorem}, we conclude that for $\theta_d$-almost every $g\in \mathbb{O}(d)$,
$$
\mathcal{L}^d \left( \left\{  a\in \mathbb{R}^d \colon \dim_H (E\cap (\tau_a \circ g)(\mathcal{P}_d)) \geq \frac{d}{2}-\varepsilon-1 =\frac{d-2}{2}-\varepsilon  \right\} \right) >0.
$$

Fix any $g\in \mathbb{O}(d)$ satisfying the above condition. Note that $\tau_a \circ g = g \circ \tau_{g^{-1}a}$ and that $g\in \mathbb{O}(d)$ preserves the Lebesgue measure $\mathcal{L}^d$, which imply
$$
\mathcal{L}^d \left( \left\{  a\in \mathbb{R}^d \colon \dim_H (E\cap (g \circ \tau_a)(\mathcal{P}_d)) \geq \frac{d-2}{2}-\varepsilon  \right\} \right) >0.
$$

Define the set $J \subset \mathbb{R}^d$ by
$$
\begin{aligned}
J&\coloneqq \left\{  a\in \mathbb{R}^d \colon \dim_H (E\cap (g \circ \tau_a)(\mathcal{P}_d)) \geq \frac{d-2}{2}-\varepsilon  \right\}\\
&= \left\{  a\in \mathbb{R}^d \colon \dim_H (g^{-1}(E)\cap (a+ \mathcal{P}_d)) \geq \frac{d-2}{2}-\varepsilon  \right\}.
\end{aligned}
$$
By definition, $\mathcal{L}^d(J)>0$. Moreover, for each $a\in J$, define $F_a = g^{-1}(E) \cap (a+\mathcal{P}_d)$. By construction, $\dim_H F_a \geq \frac{d-2}{2}-\varepsilon$ and $|\Pi(F_a)| \leq |\Pi(g^{-1}(E))| = |\Pi(E)|=0$.

\begin{remark}
Although the intersection theorem guarantees the existence of some counterexamples, it is noteworthy that applying the theorem results in the loss of $1$ dimension in the exponent, corresponding to the codimension of a $(d-1)$-dimensional hypersurface. We believe this dimensional loss is excessive, suggesting that counterexamples with larger dimension might be achievable. Additionally, we cannot precisely localize the points $a$ for which $a+\mathcal{P}_d$ carries $(\frac{d-2}{2}-\varepsilon)$-dimensional counterexamples. However, we do know that such points exist in abundance, as the set of such points $a$ has positive measure.
\end{remark}

\section{Comparison with the discrete case and conjecture}

\subsection{Comparison with the discrete case}
In \cite{Discrete_product_sets}, Che-Jui Chang, Ali Mohammadi, Thang Pham, and Chun-Yen Shen proved that assuming the extension conjecture, for an admissible dimension $d$ and a prime power $q$, if $E$ is a subset of 
$$
\mathcal{P}_d = \{(x_1,\dots,x_d) \colon x_d = x_1^2 + \cdots + x_{d-1}^2\} \subset \mathbb{F}_q^d,
$$
with $|E| \gg q^{\frac{d}{2} - \frac{1}{2d}}$, then $|\Pi(E)| \gg q$. 

In our paper, we prove that for every $a \in \mathbb{R}^d$ and every compact set $E \subset a + \mathcal{P}_d$, if the Hausdorff dimension of $E$ is greater than $\frac{d}{2} - \frac{1}{4}$ for $d = 3$, or greater than $\frac{d}{2} - \frac{1}{4} - \frac{1}{8d - 4}$ for $d \geq 4$, then $|\Pi(E)| > 0$. The key ingredient in our proof are the pinned distance results in Euclidean space, which rely on the decoupling method and the restriction theory.

Regarding negative results, the authors in \cite{Discrete_product_sets} proved that for admissible $d,q$ and any $\varepsilon>0$, there exists a subset $E\subset \mathcal{P}_d$ such that $|E| \approx q^{\frac{d-1}{2}-\varepsilon}$ and $|\Pi(E)| \approx q^{1-\varepsilon} = o(q)$. In our paper, we show that for any $\varepsilon>0$, there exists a compact set $E\subset \mathcal{P}_d$ such that 
$$
\dim_H E \geq \max\left(\frac{d-3}{2}-\varepsilon , 1 \right)
$$
and $|\Pi(E)|=0$. Here, for $3\leq d \leq 5$, we can embed the counterexample in Subsection \ref{Not study d=2} into higher-dimensional paraboloids and derive the exponent $1$, while for $d\geq 6$, we have the $\left(\frac{d-3}{2}-\varepsilon\right)$-construction. The reason why the finite field version achieves a better bound is that there exist many nonzero vectors in $\mathbb{F}_q^d$ that are mutually orthogonal. For a more precise statement, one can refer to Lemma 5.1 in Derrick Hart, Alex Iosevich, Doowon Koh, and Misha Rudnev's paper \cite{Averages_over_hyperplanes}. In contrast, in Euclidean space, no nonzero vector is self-orthogonal.

\subsection{Conjecture}

The Falconer distance conjecture states that if $E \subset \mathbb{R}^d$ is a compact set with Hausdorff dimension greater than $\tfrac{d}{2}$, then its distance set $\Delta(E)$ has positive Lebesgue measure. An even stronger conjecture, the pinned distance conjecture, asserts that $\dim_H E > \tfrac{d}{2}$ guarantees the existence of some $x \in E$ such that $|\Delta^x(E)|>0$.

Recall from our previous discussion that both $S$ and $D$ yield the exponent $\tfrac{d-1}{2}$. Moreover, the conjectured sharp exponent for the pinned distance problem is $\tfrac{d}{2}$. When studying dot products on paraboloids, we apply pinned distance results in one lower dimension, so the relevant exponent becomes $\tfrac{d-1}{2}$. Therefore, if the pinned distance conjecture is true, then for every $a \in \mathbb{R}^d$ and every compact set $E \subset a+\mathcal{P}_d$ with $\dim_H E > \tfrac{d-1}{2}$, one has $|\Pi(E)|>0$.

Motivated by this observation and by the finite field counterexample obtained in \cite{Discrete_product_sets}, we conjecture the following: for every $a \in \mathbb{R}^d$ and every compact set $E \subset a+\mathcal{P}_d$ with $\dim_H E > \tfrac{d-1}{2}$, there exists $x \in E$ such that $|\Pi^x(E)|>0$.

\bibliographystyle{plain}
\bibliography{mybibliography}

\vspace{2em}

\noindent \textsc{Chun-Kai Tseng, National Center for Theoretical Sciences, Taiwan}\\
\textit{email:} \texttt{firework1153@gmail.com}

\end{document}